\documentclass[
    leqno,
    11pt,
]{amsart} 
\usepackage{amsmath} 
\usepackage{amssymb} 
\usepackage{amsthm} 
\usepackage{dsfont}
\usepackage{mathrsfs} 
\usepackage{mathtools} 
\usepackage{float}
\usepackage{tikz} 
\usepackage{pgfplots} 
\usepackage{hyperref} 
\usepackage{nameref} 
\usepackage{xspace} 
\usepackage[ 
    nameinlink 
]{cleveref} 
\usepackage{autonum} 
\usepackage{microtype}
\usepackage{xfrac}

\RequirePackage{etex}

\hypersetup{
    colorlinks=true,
    linkcolor=blue,
    citecolor=orange,
    filecolor=magenta,      
    urlcolor=blue,
    pdftitle={Fractional Fast Widder Theory},
    }
\pgfplotsset{width=8cm,compat=1.9}

\let\oldtocsection=\tocsection
\let\oldtocsubsection=\tocsubsection
\renewcommand{\tocsection}[2]{\hspace{0em}\oldtocsection{#1}{#2}}
\renewcommand{\tocsubsection}[2]{\hspace{1em}\oldtocsubsection{#1}{#2}}
\newtheorem{Theorem}{Theorem}[section]
\newtheorem{Lemma}[Theorem]{Lemma}

\newtheorem{Corollary}[Theorem]{Corollary}
\newtheorem{Definition}[Theorem]{Definition}
\newtheorem{Remark}[Theorem]{Remark}

\calclayout 

\newcommand{\RR}{\mathbb{R}}
\newcommand{\cont}{\mathcal{C}}

\numberwithin{equation}{section} 

\begin{document}
\title{Fractional fast diffusion with initial data a Radon measure}
\date{}
\author{Jorge Ruiz-Cases}

\address{J. Ruiz-Cases.
  Departamento de Matem\'aticas, Universidad Aut\'onoma de Madrid, 28049 Madrid, Spain and Instituto de Ciencias Matemáticas ICMAT (CSIC-UAM-UCM-UC3M),
  28049 Madrid, Spain}
\email{jorge.ruizc@uam.es}

\begin{abstract}
    We establish a complete Widder Theory for the fractional fast diffusion equation. Our work focuses on nonnegative solutions satisfying a certain integral size condition at infinity. We prove that these solutions possess a Radon measure as initial trace, and prove the existence and uniqueness of solutions originating from such initial data. The uniqueness result is the main issue. Most of its difficulty comes from the singular character of the nonlinearity. 
\end{abstract}



\maketitle

\section{Introduction}

\subsection{Goal} The aim of this paper is to establish a one-to-one correspondence between optimal initial data and nonnegative very weak solutions with at most certain growth at infinity for the fractional fast diffusion equation
\begin{equation} \label{equation} \tag{E}
    \partial_t u + (-\Delta)^{\sigma/2}u^m = 0, \qquad m_c<m<1,
\end{equation}
with $\sigma \in (0,2)$, $N > \sigma$ and $m_c = \frac{N-\sigma}{N}$. In order to do so, we divide this issue into three natural questions.

The first one is to identify the optimal class of initial data for our equation. We understand the initial condition as the limit $\lim_{t \to 0^+} u(t)$ of a nonnegative solution in the vague topology. We address this in~\Cref{se:traces}, where we find that any nonnegative solution has indeed an initial trace which is a nonnegative Radon measure $\mu$ in a certain growth class $\mathfrak{M}_\sigma$. This result is true for any~{$0<m<1$}.

The next question concerns the existence of a solution for every nonnegative Radon measure in~$\mathfrak{M}_\sigma$. This question is answered in~\Cref{se:existence}, where we show that, if $m>m_c$, there exists a very weak solution of equation~\eqref{equation} with a measure $\mu \in \mathfrak{M}_\sigma$ as initial trace.

The final and most difficult question, whether nonnegative solutions starting from these Radon measures are unique or not, is settled in~\Cref{se:uniqueness}, again if $m>m_c$. This is the only section that requires $N>\sigma$, although we keep this hypothesis throughout the work for simplicity.

\subsection{Precedents}

The question of existence and uniqueness for nonnegative solutions of the local heat equation, $\partial_t u - \Delta u = 0$ was settled by Widder in~\cite{WidderPositiveTemperatures, WidderHeatEquation} in the one-dimensional case. The rest of the theory and the extension to more dimensions was done mainly by Aronson in several works, see~\cite{AronsonNonnegativeSolutions, AronsonWidderInversion, Krzyzanski,  WilcoxPositiveTemperatures}. 

In the case of the fractional heat equation (FHE), $\partial_t u + (-\Delta)^{\sigma/2} u = 0$, we refer to the works~\cite{BarriosPeralSoriaValdinoci2024_NonlocalWidder, BonforteSireVazquez2017_OptimalWidder, BonforteVazquezQuantitativeEstimates}. In~\cite{BarriosPeralSoriaValdinoci2024_NonlocalWidder}, uniqueness of nonnegative solutions is proven, and some classical results for the heat equation are extended to the fractional case. In~\cite{BonforteSireVazquez2017_OptimalWidder}, a complete Widder theory for the FHE is presented, although existence of initial traces was already shown in the previous paper~\cite{BonforteVazquezQuantitativeEstimates}. Existence of solutions is obtained like in the local case, using the corresponding integral representation. 

Concerning nonlinear equations in the local framework, we turn our attention to the \textit{porous medium equation}, $\partial_t u - \Delta u^m = 0$ for $m>0$. It is usually called \textit{slow diffusion equation} when $m>1$, \textit{fast diffusion equation} when $\frac{(N-2)_+}{N}<m<1$, and \textit{very fast diffusion equation} when $0<m \leq \frac{(N-2)_+}{N}$. It is useful to make these distinctions since much of the theory changes according to the range of $m$.

For the case $m>1$, the Widder theory was completed in several works~\cite{AronsonCaffarelli1983_InitialTrace, BenilanCrandallPierre1984_InitialValue, DahlbergKenigPorous, PierreUniquenessMeasure}. Pierre proved in~\cite{PierreUniquenessMeasure} uniqueness for nonnegative solutions arising from a finite measure. Then, Bénilan, Crandall and Pierre proved in~\cite{BenilanCrandallPierre1984_InitialValue} that there exists a solution with initial data a nonnegative measure satisfying certain growth conditions. Shortly after, Dahlberg and Kenig proved in~\cite{DahlbergKenigPorous} uniqueness for such solutions. Aronson and Caffarelli then stated in~\cite{AronsonCaffarelli1983_InitialTrace} that the admissible initial traces of solutions of the slow diffusion equation were specifically the initial conditions used in~\cite{BenilanCrandallPierre1984_InitialValue}, hence concluding the Widder theory for this equation.

When $m<1$, the main works concerning this theory are~\cite{DahlbergKenigFast, HerreroPierreFastRegularizing}. Herrero and Pierre established in~\cite{HerreroPierreFastRegularizing} a regularizing effect from $L^1_{\text{loc}}$ to $L^\infty_{\text{loc}}$. Here it is particularly important that $m> \frac{N-2}{N}$ when $N>2$, since it is known that for $0<m\leq \frac{N-2}{N}$ such a regularizing effect cannot hold for general~$L^1_\text{loc}$ initial data, see~\cite{BenilanCrandallPierre1984_InitialValue}. This result is used and extended by Dahlberg and Kenig in~\cite{DahlbergKenigFast}, where they present the complete Widder theory for the fast diffusion equation. In fact the equation they study in that work is $\partial_t u - \Delta \varphi(u) = 0$, where the nonlinearity $\varphi$ is a generalization of the power~$\varphi(s) \sim  s^m$ for~$\frac{N-2}{N} < m <1$. This last paper is one of the biggest sources of inspiration for the present one, and so is~\cite{DahlbergKenigGeneralizedPorous}, where the same authors work on a generalized slow diffusion equation.

In the nonlocal framework, for the \textit{fractional porous medium equation} $\partial_t u + (-\Delta)^{\sigma/2} u^m = 0$, with~$m>0$, the case of $L^1$ initial data is now well understood, see for instance~\cite{delTesoEndalEspen2017_DistUniqueness, GrilloMuratoriPunzo2020_VeryWeak, dePabloQuiros2011, dePabloQuiros2012}. That solutions starting from a Dirac delta as an initial condition are unique if $m>m_c$ was proved by Vázquez in~\cite{VazquezBarenblattnolocal}. In that same paper, he mentions that the question of uniqueness for solutions with a general Radon measure as initial condition seems to be nontrivial. Existence of weak energy solutions starting from a finite measure when $m>m_c$ is also addressed in that paper. 

With respect to the fractional slow diffusion equation, Grillo, Muratori and Punzo prove uniqueness of nonnegative weak solutions arising from a nonnegative finite Radon measure, see~\cite{GrilloMuratoriPunzo2015_Measure}. Their proof adapts to the nonlocal framework the result by Pierre in~\cite{PierreUniquenessMeasure}. They also treat the additional difficulty of having a singular weight in front of the time derivative of the equation. 

For the fractional fast diffusion equation, Bonforte and Vázquez in~\cite{BonforteVazquezQuantitativeEstimates} work with a subclass of very weak solutions, which we will call \textit{weighted solutions}, and start developing their theory, including many estimates and results useful for this work. In that same paper they show what the initial traces are for these solutions in the whole range $m>0$.

\subsection{Main results}

This paper can be regarded as a direct sequel to~\cite{BonforteVazquezQuantitativeEstimates, DahlbergKenigFast, GrilloMuratoriPunzo2015_Measure}. It is a sequel to~\cite{GrilloMuratoriPunzo2015_Measure}, since the result stated there works for the fractional slow diffusion equation and here we complete the analogous for the fractional fast diffusion equation, although using a more general concept of solution. It is a sequel to~\cite{BonforteVazquezQuantitativeEstimates} in the sense that we use and build upon the theory of weighted solutions established in that work. And it is especially a sequel to~\cite{DahlbergKenigFast}, since there the Widder theory for the local fast diffusion equation is presented, and we aim to adapt the same results to the nonlocal setting.

We start our work by defining a \textit{weighted solution} as a very weak solution satisfying a certain growth at infinity. Then, a new estimate is proven in~\Cref{re:weightedcomparison}, which gives a comparison principle, and thus uniqueness, of solutions for initial conditions that present the previously mentioned growth at infinity. This result is valid for sign-changing solutions and allows us to prove existence (with initial data a function) through an approximation. To the best of our knowledge, these results have not been previously established in the literature. 

Afterwards, in~\Cref{re:initialtraces}, the existence of initial traces is shown, which is essentially Theorem~7.2 of~\cite{BonforteVazquezQuantitativeEstimates} and is included here for the sake of completeness.

We prove existence of weighted solutions starting from a Radon measure in $\mathfrak{M}_\sigma$ in~\Cref{re:existencia}. Here we require an approximation result for solutions with initial data a measure,~\Cref{re:measurecomparison}, that will also be of relevance in the uniqueness result.

Finally, we show uniqueness of solutions. The first important result is~\Cref{re:firstuniqueness}, where most of the technical difficulties of this work can be found. This is the point where we require the restriction~${N>\sigma}$, since a requisite for our proof is the use of the Riesz potential. We prove that nonnegative weak solutions starting from a finite nonnegative measure are unique. This works for the whole range $m>0$ with very little changes, so it recovers the uniqueness result for the fractional slow diffusion equation from~\cite{GrilloMuratoriPunzo2015_Measure}. As a corollary, we prove the main result of this work,~\Cref{re:mainuniqueness}. In it, we show uniqueness of weighted solutions starting from an initial condition in~$\mathfrak{M}_\sigma$. This result is an adaptation of~\cite[Theorem 3.28]{DahlbergKenigFast} to the nonlocal framework. We avoid the need to study~$L^1$--$L^\infty$ smoothing effects like the one in~\cite{HerreroPierreFastRegularizing} and their extensions in~\cite{DahlbergKenigFast}, since in~\cite[Theorem 3.28]{DahlbergKenigFast} they used very weak solutions as an approximation, and here we will use weak solutions instead.

The uniqueness proof for the fast diffusion case is more complicated than both the linear and the slow diffusion cases and, in a way, encapsulates them both. The reason is that a certain coefficient function $\alpha$ appears, which in the linear case is $\alpha = 1$, and in the slow diffusion case $\alpha$ may be degenerate. However, in the fast diffusion case, the coefficient can exhibit both degenerate and singular behavior. 

\subsection{Outline of the paper}

In~\Cref{se:preliminaries} we introduce some required concepts and helpful tools already available in the literature. In~\Cref{se:weighted}, we define \textit{weighted solutions} and prove useful estimates that they satisfy.~\Cref{se:traces} is concerned with the existence of initial traces. In~\Cref{se:existence}, we prove existence of weighted solutions starting from an initial condition in $\mathfrak{M}_\sigma$.~\Cref{se:uniqueness} contains the main result of this paper, which is the uniqueness of nonnegative solutions. Comments and possible extensions are mentioned in~\Cref{se:extensions}.

\section{Preliminaries} \label{se:preliminaries}

\subsection{On the fractional Laplacian and fractional Sobolev spaces}

The nonlocal operator~$(-\Delta)^{\sigma/2}$, known as the \textit{fractional Laplacian} of order~$\sigma \in (0,2)$, is defined for any function~$g$ in the Schwartz class via the Fourier transform,
\begin{equation}
    \widehat{(-\Delta)^{\sigma/2} g}(\xi) = |\xi|^\sigma \widehat{g}(\xi).
\end{equation}
It can also be written by means of a hypersingular kernel, 
\begin{equation} \label{eq:fractionallaplacian}
    (-\Delta)^{\sigma/2} g(x) = C_{N,\sigma} \text{P.V.} \int_{\RR^N} \frac{g(x) - g(y)}{|x-y|^{N+\sigma}}\, dy, \qquad C_{N,\sigma} = \frac{2^{\sigma-1}|\sigma| \Gamma((N+\sigma)/2)}{\pi^{N/2} \Gamma(1 - \sigma/2)},
\end{equation}
where P.V.\ stands for \textit{principal value}. The constant $C_{N,\sigma}$ can be found in \cite{Steinconstant}, for instance. The natural energy space that appears when working with weak solutions associated to this operator is the \textit{fractional homogeneous Sobolev space} $\dot{H}^{\sigma/2}(\RR^N)$, defined as the completion of $\cont^\infty_c(\RR^N)$ with the seminorm
\begin{equation}
    [\phi]_{\dot{H}^{\sigma/2}} = \left( \int_{\RR^N} |\xi|^{\sigma} |\widehat{\phi}|^2 \right)^{1/2} = \|(-\Delta)^{\sigma/4} \phi\|_2.
\end{equation}
The corresponding bilinear form can be expressed as
\begin{equation}
    \mathcal{E}(\phi,\psi) = \int_{\RR^N}(-\Delta)^{\sigma/4} \phi (-\Delta)^{\sigma/4} \psi = \frac{C_{N,\sigma}}{2} \int_{\RR^N}\int_{\RR^N} \frac{(\phi(x)-\phi(y)) (\psi(x) - \psi(y))}{|x-y|^{N+\sigma}}\, dxdy;
\end{equation}
hence, an alternative way to write the $\dot{H}^{\sigma/2}(\RR^N)$ seminorm is
\begin{equation}
    \|(-\Delta)^{\sigma/4} \phi\|_2 = \overline{\mathcal{E}}^{1/2}(\phi) = \frac{C_{N,\sigma}}{2} \left( \int_{\RR^N}\int_{\RR^N} \frac{(\phi(x)-\phi(y))^2}{|x-y|^{N+\sigma}}\, dxdy \right)^{1/2}. 
\end{equation}
A useful inequality to bear in mind for every $\phi \in \dot{H}^{\sigma/2}(\RR^N)$ is the Hardy-Littlewood-Sobolev inequality
\begin{equation} \label{eq:hlsinequality}
    \|\phi\|_{ L^{\frac{2N}{N-\sigma} }(\RR^N)} \leq C [\phi]_{\dot{H}^{\sigma/2}(\RR^N)} \quad \text{ if } N > \sigma,
\end{equation}
where the constant depends only on $N$ and $\sigma$. In fact, elements of the fractional homogeneous Sobolev space $\dot{H}^{\sigma/2}(\RR^N)$ are $L^{\frac{2N}{N-\sigma}}$ functions with finite $\dot{H}^{\sigma/2}$ seminorm.

\subsection{On the concepts of solution}

Let us start with the most studied concept of solution, which is the standard \textit{weak solution}. Let $Q = \RR^N \times (0, \infty)$ and let $L^p_\vartheta = L^p(\RR^N, \vartheta \,dx)$ for a weight function~$\vartheta$. 

\begin{Definition} \label{def:weaksolutions}
    A \textit{weak solution} or \textit{weak $L^1$ energy solution} of equation~\eqref{equation} is a function ${u \in \cont((0,\infty) \! : \! L^1(\RR^N))}$ with ${u^m \in L^2_\text{loc}((0, \infty) \! : \! \dot{H}^{\sigma/2}(\RR^N))}$ such that 
    \begin{equation}\label{weakuation}
        \int_0^\infty \int_{\RR^N} u \partial_t \zeta \, dxdt - \int_0^\infty \int_{\RR^N} (-\Delta)^{\sigma/4}u^m (-\Delta)^{\sigma/4} \zeta \, dxdt = 0
    \end{equation}
    for every test function $\zeta \in \cont^\infty_c(Q)$.
\end{Definition}

Existence, uniqueness and most of the basic theory for solutions with initial data~${u_0 \in L^1(\RR^N)}$ have already been developed, see~\cite{dePabloQuiros2012}. Another concept of solution is that of \textit{very weak solution}.

\begin{Definition} \label{def:veryweaksolutions}
    A \textit{very weak solution} of equation~\eqref{equation} is a function ${u \in L^1_\text{loc}((0,\infty) \! : \! L^1_\text{loc}(\RR^N))}$ with ${u^m \in L^1_\text{loc}((0, \infty) \! : \! L^1_\rho)}$, where $\rho(x) = (1+|x|^2)^{-(N+\sigma)/2}$, such that
    \begin{equation}\label{veryweakuation}
        \int_0^\infty \int_{\RR^N} u \partial_t \zeta \, dxdt - \int_0^\infty\int_{\RR^N} u^m (-\Delta)^{\sigma/2}\zeta \,dxdt = 0
    \end{equation}
    for every test function $\zeta \in \cont^\infty_c(Q)$.
\end{Definition} 

Notice that weak solutions are also very weak solutions thanks to~\eqref{eq:hlsinequality}. They also satisfy the following identity,
\begin{equation} \label{eq:veryweakwithinitialdata}
    \int_{\RR^N} u(x,T) \xi(x,T) \,dx -  \int_{\RR^N} u(x,\tau) \xi(x,\tau) \,dx = \int_\tau^T \int_{\RR^N} \big(u \partial_t \xi - u^m (-\Delta)^{\sigma/2}  \xi\big) \, dx dt,
\end{equation}
for almost every $\tau, T>0$, with $\xi \in C^\infty_c(Q_T^\tau)$, where $Q_T^\tau = \RR^N \times [\tau, T]$. It will be particularly useful throughout this work to consider the previous expression for a test function without time dependence, that is,
\begin{equation}\label{eq:salto_tiempo}
    \int_{\RR^N} (u(x, t+h)-u(x, t)) \phi(x) \,dx = - \int_t^{t+h} \int_{\RR^N}  u^m(x,s) (-\Delta)^{\sigma/2}\phi(x)  \,dx ds,
\end{equation}
for every $\phi \in \cont^\infty_c(\RR^N)$ and almost every $t, h \geq 0$. 
 
Regarding boundedness of solutions, under the condition that $m>m_c$, weak solutions satisfy an~$L^1$--$L^\infty$ \textit{smoothing effect}, see~\cite{dePabloQuiros2012}.

\begin{Theorem}[Smoothing effect]\label{re:smoothing_1}
    Let $m>m_c$. Assume $u$, a weak solution of~\eqref{equation}, verifies $\lim_{t\to 0}u(t) = u_0$ in $L^1(\RR^N)$. Then,
    \begin{equation} \label{eq:smoothingeq}
        \|u(t)\|_\infty \leq C t^{-\alpha}\|u_0\|_{1}^{\gamma} \quad \text{ for all } t>0,
    \end{equation}
    with $\alpha = N /(N(m-1)+\sigma)$ and $\gamma = \sigma \alpha / N$, where the constant $C$ depends only on $m, N$, and $\sigma$.
\end{Theorem}

\subsection{On the Radon measures}

A signed measure $\mu$ is called a \textit{Radon measure} if it is Borel regular and locally finite. The space of Radon measures on $\RR^N$, denoted by $\mathfrak{M}(\RR^N)$, is the dual space of~$\cont_c(\RR^N)$. Inspired by this fact, a natural convergence in~$\mathfrak{M}(\RR^N)$ is the weak* convergence, defined by the rule $\mu_n \rightharpoonup \mu$ if and only if
\begin{equation}
    \lim_{n \to \infty} \int_{\RR^N} \phi \, d\mu_n = \int_{\RR^N} \phi \, d\mu_n
\end{equation}
for any $\phi \in \cont_c(\RR^N)$. This type of convergence is also called \textit{vague convergence}. The associated topology is described as $\sigma(\mathfrak{M}(\RR^N) , \cont_c(\RR^N))$ and is known as the \textit{vague topology}.

Let us denote by $\mathfrak{M}^+(\RR^N)$ the set of nonnegative Radon measures on $\RR^N$. We will need the following compactness result for nonnegative measures, see for instance~\cite[Theorem 0.6]{LandkofPotentialbook}. 

\begin{Theorem} \label{re:measurecompactness}
    Let $\{\mu_n\}_n \subset \mathfrak{M}^+(\RR^N)$ satisfy
    \begin{equation}
        \sup_n \mu_n(K) < \infty \text{ for any compact set } K\subset \RR^N.
    \end{equation}
    Then there exists a subsequence $\{\mu_{n_k}\}_k$ and $\mu \in \mathfrak{M}^+(\RR^N)$ such that $\mu_{n_k} \rightharpoonup \mu$.
\end{Theorem}

\subsection{Potential theory and capacity}
Most of the theory and results in this subsection come from~\cite{LandkofPotentialbook}. We include them here for the reader's convenience. Recall that we will always assume~${N>\sigma}$. 

Let us define the \textit{Riesz potential} as
\begin{equation}
    \mathcal{I}_\sigma(x) = \frac{C_{N,-\sigma}}{|x|^{N-\sigma}},
\end{equation}
where $C_{N, -\sigma}$ is as in~\eqref{eq:fractionallaplacian}. We have that $\mathcal{I}_\sigma \in L^{\frac{N}{N-\sigma}, \infty}(\RR^N)$, the weak $L^p$ space. Define the potential of a function $u$ as $U^u_\sigma(x) = (\mathcal{I}_\sigma * u)(x)$. Note that if $u$ is regular enough, it holds that~${(-\Delta)^{\sigma/2}U^u_\sigma = u}$. The \textit{Riesz potential of order $\sigma$} of a signed measure $\nu$, or simply, the \textit{potential} of $\nu$ is
\begin{equation}
    U^\nu_\sigma(x)  = C_{N,-\sigma}\int_{\RR^N} \frac{d\nu(y)}{|x-y|^{N-\sigma}},
\end{equation}
formally the convolution $\mathcal{I}_\sigma * \nu$. 

The following result will be of help when proving uniqueness. 
\begin{Theorem} \label{re:potentialuniqueness}
    If the potential of a signed measure $\nu$ is equal to zero almost everywhere, then $\nu \equiv 0$ and consequently, $U^\nu_\alpha (x) \equiv 0$ everywhere.
\end{Theorem}

It is also useful to know that potentials of a measure have some regularity.
\begin{Remark} \label{re:mark}
    Potentials $U^\nu_\sigma(x)$ satisfy 
    \begin{equation}
        U^\nu_\sigma(x) = \liminf_{y \to x} U^\nu_\sigma(y),
    \end{equation}
    and therefore they are lower semicontinuous. 
\end{Remark}

We now remind the reader about the idea of \textit{capacity}. Let $K \subset \RR^N$ be a compact set and consider the set $\dot{\mathfrak{M}}^+(K)$ of nonnegative Radon measures $\mu$ whose support lies in $K$, and such that $\mu(\RR^N) = 1$. This set is convex and compact in the vague topology. Then, we define
\begin{equation}
    W_\sigma(K) = \inf_{\mu \in \dot{\mathfrak{M}}^+(K)} \left\{\int_{K\times K} \mathcal{I}_\sigma(x-y) \,d\mu(x) d\mu(y)\right\}.
\end{equation}
\begin{Definition}
    The capacity (or $\sigma$-capacity) of a compact set $K$ is 
    \begin{equation}
        C_\sigma(K) = \frac{1}{W_\sigma(K)}.
    \end{equation}
\end{Definition}
We now extend this definition to arbitrary sets.
\begin{Definition}
    For an arbitrary set $E \subset \RR^N$, we define the \textit{inner capacity} $\underline{C}_\sigma(E)$ as 
    \begin{equation}
        \underline{C}_\sigma(E) = \sup\{ C_\sigma(K) \!: K \subset E \text{ is compact} \},
    \end{equation}
    and the \textit{outer capacity} $\overline{C}_\sigma(E)$ as 
    \begin{equation}
        \overline{C}_\sigma(E) = \inf \{ \underline{C}_\sigma(G)\!  : G \supset E \text{ is open} \}.
    \end{equation}
\end{Definition}
With this in mind, let us define a particular set of measures that will be helpful to us.
\begin{Definition}
    A \textit{C-absolutely continuous measure} $\mu$ is a measure such that if a set $E$ has inner capacity zero, then $\mu(E) = 0$.  
\end{Definition}
Hence, when integrating against C-absolutely continuous measures, one may ignore any set with inner capacity zero. A useful criterion to determine whether a measure is C-absolutely continuous is as follows. 
\begin{Lemma} \label{re:criteriocabsolutely}
    If $U^\mu_\sigma(x) < \infty$ for all $x\in \RR^N$, then the measure $\mu$ is C-absolutely continuous.
\end{Lemma}
A property is verified \textit{quasi-everywhere} if it is satisfied in $\RR^N \setminus E$, where $E$ is a set of zero outer capacity. By definition, if $E$ has outer capacity zero, then it has inner capacity zero. Thus, if a property is satisfied quasi-everywhere, we can assume it is satisfied everywhere when integrating against a C-absolutely continuous measure. This will be of importance when applying the next result. 
\begin{Theorem} \label{re:limpotential}
    Let $\mu_n \in \mathfrak{M}^+(\RR^N)$, $\mu_n \rightharpoonup  \mu$ and
    \begin{equation} \label{eq:controlmasa}
        \lim_{r\to \infty} \int\limits_{|x|>r} \frac{d\mu_n(x)}{|x|^{N-\sigma}} = 0
    \end{equation}
    uniformly with respect to $n$. Then,
    \begin{equation}
        U^\mu_\sigma(x) = \liminf_{n\to \infty} U^{\mu_n}_\sigma(x) \quad \text{quasi-everywhere}.
    \end{equation}
\end{Theorem}

\begin{Remark}
    Condition~\eqref{eq:controlmasa} is satisfied trivially if there exists $M > 0$ such that $\mu_n(\RR^N) < M$ for all $n \in \mathbb{N}$.
\end{Remark}

The last result that we need is the one below.
\begin{Theorem} \label{re:potentialmonotoneconv}
    Let $\{U^{\mu_n}_\sigma\}_n$ be a monotone decreasing sequence of potentials. Then, there is a measure $\mu$ such that $\mu_n \rightharpoonup \mu$ and $\lim_{n\to \infty} U^{\mu_n}_\sigma(x) = U^{\mu}_\sigma(x)$ quasi-everywhere.
\end{Theorem}

\subsection{Banach-Saks: On weak and pointwise convergence}

We state here a precursor of Mazur's Theorem on strongly convergent sequences built out of weakly convergent ones, given by Banach 
and Saks in~\cite{BanachSaks}.
\begin{Theorem}[Banach-Saks] \label{re:BanachSaks}
    Let $f_n \in L^p(\Omega)$ for some $1<p< \infty$ and $\Omega \subset \RR^N$. If $f_n$ converges weakly to $f$ in $L^p(\Omega)$, then there exists a subsequence $f_{n_j}$ such that $\frac{1}{N} \sum_{j = 1}^N f_{n_j}$ converges strongly to $f$ in $L^p(\Omega)$.
\end{Theorem}
Given a sequence $\{a_j\}_j \subset \RR$, we call $\frac{1}{N} \sum_{j = 1}^N a_j$ its \textit{Cesàro mean}. Observe that if $a_n$ converges to~$a \in \RR$, then its Cesàro mean converges to $a$ too. 

In the study of PDEs, it is highly useful to have tools that determine when the weak limit and the pointwise limit of a sequence coincide. One way to establish this comes as a consequence of the previous result
\begin{Corollary}
    Let $f_n \in L^p(\Omega)$ for some $1<p< \infty$ and $\Omega \subset \RR^N$. Suppose that $\{f_n\}_n$ is bounded in $L^p(\Omega)$. If $f_n$ converges pointwise a.e.\ to $f$ then it converges weakly to $f$ in $L^p(\Omega)$.
\end{Corollary}

\section{Solutions growing at infinity} \label{se:weighted}

This section is devoted to presenting estimates for very weak solutions, and is heavily inspired by~\cite{BonforteVazquezQuantitativeEstimates}. There, a subclass of very weak solutions that we shall call \textit{weighted solutions} is studied. At the end, we will establish a comparison principle, and thus uniqueness, for these solutions. Throughout this whole section we work with sign-changing solutions, where we consider $c^m = |c|^{m-1}c$ for $c \in \RR$.

We want to study very weak solutions. According to its definition, these solutions verify 
\begin{equation}
    \int_{\RR^N} \frac{u^m(x,t)}{(1+|x|^2)^{\frac{N+\sigma}{2}}} \, dx < \infty
\end{equation}
for a.e.\ $t>0$. Therefore, $u^m(x,t)$ can present growth at infinity, as much growth as the previous integral allows. Intuitively, and following the example of the linear case~\cite{BonforteSireVazquez2017_OptimalWidder}, it is reasonable to ask for initial data~$u_0$ that satisfy 
\begin{equation} \label{eq:condpesom}
    \int_{\RR^N} \frac{u_0^m(x)}{(1+|x|^2)^{\frac{N+\sigma}{2}}} \, dx < \infty.
\end{equation}
However, we would like the condition for the initial data to be given in terms of $u_0$ and not in terms of $u_0^m$, that is, we would like to give a condition as
\begin{equation} \label{eq:condpesopeso}
    \int_{\RR^N} u_0 \vartheta < \infty
\end{equation}
for an appropriate weight $\vartheta$. The reason is that it is straightforward to change $u_0$ to a measure, which is one of the aims of this paper. Another reason is that we cannot guarantee that estimates that we will prove are satisfied if we ask for~\eqref{eq:condpesom} instead of~\eqref{eq:condpesopeso}.

Let us define the set of weights $\vartheta$. We want them to decay at infinity to allow for $u$ to grow. We define $\varTheta$ as
\begin{equation} \label{eq:varTheta} 
    \varTheta = \left\{\frac{1}{(1+(|x|^2 - 1)_+^4)^{a/8}}, \quad \text{ for } a \in \left(N, N+\frac{\sigma}{m}\right) \right\}.
\end{equation}
For any $\vartheta \in \varTheta$, we get that 
\begin{equation} \label{eq:quotient}
    |(-\Delta)^{\sigma/2}\vartheta(x)| \leq \frac{C}{1+|x|^{N+\sigma}}, \qquad \int_{\RR^N} \left(\frac{1}{(1+|x|^{N+\sigma})\vartheta^m(x)} \right)^{\frac{1}{1-m}} \, dx < \infty,
\end{equation}
for some $C>0$, as shown in~\cite{BonforteVazquezQuantitativeEstimates}. This condition is useful for the following computation,
\begin{equation} \label{eq:comprecurrente}
    \begin{aligned}
            \int_{\RR^N} u^m (-\Delta)^{\sigma/2}\vartheta &= \int_{\RR^N} u^m \vartheta^m \frac{|(-\Delta)^{\sigma/2}\vartheta|}{\vartheta^m}  \\
            &\leq \left(\int_{\RR^N} u \vartheta\right)^{m}  \left(\int_{\RR^N} \left(\frac{|(-\Delta)^{\sigma/2}\vartheta(x)|}{\vartheta(x)^m} \right)^{\frac{1}{1-m}}\right)^{1-m},
    \end{aligned}
\end{equation}
which appears several times in this work. The same computation, changing $(-\Delta)^{\sigma/2}\vartheta$ with~${(-\Delta)^{\sigma/2}\phi}$ for $\phi \in \cont^\infty_c(\RR^N)$ but still multiplying and dividing by $\vartheta^m$, is the reason why weighted solutions are a natural subclass of very weak solutions.

The set of weights $\varTheta$ defined in~\eqref{eq:varTheta} is not the most general we can work with. Notice we cannot choose $a = N+\frac{\sigma}{m}$, otherwise the second condition of~\eqref{eq:quotient} and computation~\eqref{eq:comprecurrente} would fail. However, we could add a logarithmic correction and consider weights with $a = N+\frac{\sigma}{m}$, for instance, $\log^p(|x|)|x|^{-(N+\frac{\sigma}{m})}$ for $p> \frac{1-m}{m}$ and $|x| \gg 1$.  What we want are positive weights satisfying~\eqref{eq:quotient} having the strongest decay possible, but we only work with weights $\vartheta$ with polynomial decay for the sake of simplicity. They have to be positive in order for~\eqref{eq:comprecurrente} to work, which is not needed in the local case, see~\cite{HerreroPierreFastRegularizing}. This is a restriction that appears due to the nonlocality of the operator. 

If we imposed the additional restriction that $u_0$ has at most polynomial growth, we would find that~\eqref{eq:condpesom} and~\eqref{eq:condpesopeso} with $\vartheta \in \varTheta$ are equivalent. However, in general, condition~\eqref{eq:condpesopeso} with $\vartheta \in \varTheta$ is more restrictive than~\eqref{eq:condpesom}. For a wider discussion on the initial data, see~\Cref{se:extensions}.

\begin{Definition}[Weighted solutions]
    We say $u$ is a \textit{weighted very weak solution} or, simply, \textit{weighted solution}, of equation~\eqref{equation} in $\RR^N \times (0,T)$ if it is a very weak solution such that ${u \in L^1_\text{loc}((0,T) \! : \! L^1_\vartheta)}$ for some~$\vartheta \in \varTheta$. 
\end{Definition}

Later, we will show that we can actually assume ${u \in \cont((0,T) \! : \!L^1_\vartheta)}$ in the previous definition, see~\Cref{re:continuityremark}. When comparing two weighted solutions $u,v$ with weights $\vartheta_u$ and $\vartheta_v$ respectively, we will assume without further mention that they integrate against the same $\vartheta \in \varTheta$, the weight with the stronger decay, that is, the one with the bigger exponent $a$. 

\begin{Remark}
    \emph{(i)} The fractional fast diffusion equation allows solutions to grow more than solutions of the fractional heat equation. This can be seen in the definition~\eqref{eq:varTheta} of $\varTheta$. The decay for the weight in the linear case corresponds to $a = N+\sigma$, while the decay allowed in the fast diffusion case can reach~{$N+\sigma < a < N+ \frac{\sigma}{m}$}. 

    \noindent\emph{(ii)} We can use the weights $\vartheta \in \varTheta$ as test functions in expression~\eqref{eq:salto_tiempo} for weighted solutions. This can be proven by studying the limit $\lim_{R \to \infty} (-\Delta)^{\sigma/2}(\vartheta \phi_R)$, where $\phi_R$ is a radial, smooth, compactly supported cut-off function such that $\phi_R \equiv 1$ in $B_R$, decreases in $B_{2R}\setminus B_R$ and $\phi_R\equiv 0$ in $\RR^N \setminus B_{2R}$.  
\end{Remark}

An important estimate satisfied by this kind of solution is the next one, again from~\cite{BonforteVazquezQuantitativeEstimates}.

\begin{Theorem}[Weighted $L^1$ estimates] \label{re:orderedcomparison}
    Let $u\geq v$ be two ordered weighted solutions to~\eqref{equation}, with $0 < m < 1$. Let $\vartheta_R(x) = \vartheta(x/R)$ where $R>0$ and $\vartheta \in \varTheta$. Then, for all $0 \leq \tau, t < \infty$ we have
    \begin{equation} \label{eq:orderedcomparison}
        \left( \int_{\RR^N} (u-v)(x,t) \vartheta_R(x) \, dx\right)^{1-m} \leq \left( \int_{\RR^N} (u-v)(x,\tau ) \vartheta_R(x) \, dx\right)^{1-m} + \frac{C |t-\tau| }{R^{\sigma - N(1-m)}},
    \end{equation}
    with $C>0$ depending only on $\sigma$, $m$ and $N$.
\end{Theorem}

\begin{Remark} \label{re:acotacionhastacero}
    As a consequence, if $u$ is a nonnegative weighted solution, then~${u \in L^\infty((0, \infty) \! : \! L^1_\vartheta)}$.
\end{Remark}

Let us prove a stronger result applicable to any two weighted solutions, not necessarily ordered and maybe changing sign, that will give us a comparison principle. To achieve this, we require a \textit{Kato type inequality}: let $f$ be a smooth, convex function; then, 
\begin{equation} \label{eq:smoothkato}
    (-\Delta)^{\sigma/2} f(\phi) \leq f'(\phi) (-\Delta)^{\sigma/2} \phi,
\end{equation}
which in particular implies 
\begin{equation} \label{eq:positivekato}
    (-\Delta)^{\sigma/2} \phi_+ \leq \chi_{\{\phi\geq 0\}} (-\Delta)^{\sigma/2} \phi,
\end{equation}
see for instance~\cite{CordobaCordobaKatoInequality}. Let us explain the formal argument: we have that, for $w = u-v$,
\begin{equation}
    (-\Delta)^{\sigma/2} (u^m - v^m)_+ \leq \chi_{\{w\geq 0\}} (-\Delta)^{\sigma/2} (u^m - v^m).
\end{equation}
With this in mind and observing that
\begin{equation}
    \partial_t w_+ = \chi_{\{w\geq 0\}}\partial_t w,
\end{equation}
we get, using the equation, that
\begin{equation}
    \partial_t w_+ = -\chi_{\{w\geq 0\}} (-\Delta)^{\sigma/2} (u^m - v^m) \leq - (-\Delta)^{\sigma/2} (u^m - v^m)_+.
\end{equation}
Multiplying this expression by $\vartheta_R$ and integrating we obtain
\begin{equation}
    \begin{aligned}
        \frac{d}{dt} \int_{\RR^N} w_+ \vartheta_R & \leq - \int_{\RR^N} (-\Delta)^{\sigma/2} (u^m - v^m)_+ \vartheta_R = - \int_{\RR^N}  (u^m - v^m)_+ (-\Delta)^{\sigma/2}\vartheta_R    \\
                                                & \leq \int_{\RR^N}  (u^m - v^m)_+ |(-\Delta)^{\sigma/2}\vartheta_R| \leq 2^{1-m} \int_{\RR^N}  [(u - v)_+]^m |(-\Delta)^{\sigma/2}\vartheta_R| \\
                                                & = \int_{\RR^N}  (w_+)^m |(-\Delta)^{\sigma/2}\vartheta_R|.
    \end{aligned}
\end{equation}
Notice that we used $(a^m-b^m) \leq 2^{1-m}(a-b)^m$, which is true for $a \geq b$ and for any $0<m<1$. Observe that if in this last term we apply~\eqref{eq:comprecurrente}, we end up with the differential inequality
\begin{equation}
    \frac{d}{dt} \int_{\RR^N} w_+ \vartheta_R \leq C\left(\int_{\RR^N} w_+ \vartheta_R\right)^m, 
\end{equation}
with $C>0$ depending only on $\sigma$, $m$ and $N$, just like in~\Cref{re:orderedcomparison}. From this, we conclude that
\begin{equation}
    \left( \int w_+(x,t) \vartheta_R(x) \, dx\right)^{1-m} \leq \left( \int w_+(x,\tau ) \vartheta_R(x) \, dx\right)^{1-m} + \frac{C |t-\tau| }{R^{\sigma - N(1-m)}}
\end{equation}
for $0\leq \tau \leq t < \infty$. This estimate holds for any $0<m<1$. Assume,  additionally, that $m>m_c$, so that $\sigma - N(1-m) >0$. Then, for $u(0) \leq v(0)$ we can take the limit $R\to \infty$ to obtain that
\begin{equation}
    \int_{\RR^N} w_+ \leq 0,
\end{equation}
which implies that $w_+ = 0$ a.e.\ and thus $u \leq v$.

\begin{Theorem}[Comparison estimate] \label{re:weightedcomparison}
    Let $u,v$ be two weighted solutions of equation~\eqref{equation}. Then,
    \begin{equation} \label{eq:weightedcomparison}
        \left( \int_{\RR^N} (u-v)_+(x,t) \vartheta_R(x) \, dx\right)^{1-m} \leq \left( \int_{\RR^N} (u-v)_+(x,\tau ) \vartheta_R(x) \, dx\right)^{1-m} + \frac{C (t-\tau) }{R^{\sigma - N(1-m)}}
    \end{equation}
    for all $0 \leq \tau \leq t < \infty$. If, in addition, we assume $m_c<m<1$, then $u(\tau) \leq v(\tau)$ implies $u(t) \leq v(t)$ for any $t\geq \tau$.
\end{Theorem}
\begin{proof}
    Given that $u$ and $v$ are very weak solutions, they satisfy~\eqref{veryweakuation}, hence
    \begin{equation}
        \int_a^b \int_{\RR^N} (u-v) \partial_t \xi  = \int_a^b \int_{\RR^N} (u^m - v^m) (- \Delta)^{\sigma/2} \xi ,
    \end{equation}
    for $0\leq a \leq b$. Since $u, v \in L^1\left(\RR^N \times (a,b) , \vartheta \, dx \right)$, we can choose $u_\varepsilon, v_\varepsilon \in \cont_c^\infty(\RR^N \times (a,b))$ such that
    \begin{equation}
        \int_a^b \int_{\RR^N} \left|u - u_\varepsilon \right| \vartheta  < \frac{\varepsilon}{2}, \qquad
        \int_a^b \int_{\RR^N} \left|v - v_\varepsilon \right| \vartheta  < \frac{\varepsilon}{2}.
    \end{equation}
    Notice that $u_\varepsilon^m$ and $v_\varepsilon^m$ also approximate $u^m$ and $v^m$, respectively, in $L^1\left(\RR^N \times (a,b) , (1+|x|)^{-(N + \sigma)}dx \right)$, thanks to~\eqref{eq:comprecurrente}.

    Now choose $f_j$ a smooth, convex approximation of the sign function and take $\xi = f'_j(u_\varepsilon - v_\varepsilon) \widetilde{\xi}$, with $0 \leq \widetilde{\xi} \in \cont_c^\infty(\RR^N \times (a,b))$. For the left-hand side, we obtain
    \begin{equation}
        \begin{aligned}
            \int_a^b \int_{\RR^N} (u-v) \partial_t \big( f'_j(u_\varepsilon - v_\varepsilon)\widetilde{\xi} \big)  & \leq \int_a^b \int_{\RR^N} (u_\varepsilon-v_\varepsilon) \partial_t\big( f'_j(u_\varepsilon - v_\varepsilon)\widetilde{\xi} \big)   + \varepsilon \\
            &= -\int_a^b \int_{\RR^N} \partial_t(u_\varepsilon-v_\varepsilon)  f'_j(u_\varepsilon - v_\varepsilon)\widetilde{\xi}    + \varepsilon \\
            & = - \int_a^b \int_{\RR^N} \partial_t f_j(u_\varepsilon - v_\varepsilon) \widetilde{\xi}  + \varepsilon                                  \\
                                                                                                                          & = \int_a^b \int_{\RR^N} f_j(u_\varepsilon - v_\varepsilon) \partial_t \widetilde{\xi}  + \varepsilon. 
        \end{aligned}
    \end{equation}
    For the right-hand side, we get that
    \begin{equation}
        \begin{aligned}
            \int_a^b \int_{\RR^N} (u^m-v^m) (-\Delta)^{\sigma/2} \big( f'_j(u_\varepsilon - v_\varepsilon)\widetilde{\xi} \big)  & \geq \int_a^b \int_{\RR^N} (u_\varepsilon^m-v_\varepsilon^m) (-\Delta)^{\sigma/2} \big( f'_j(u_\varepsilon - v_\varepsilon)\widetilde{\xi} \big)  - \varepsilon \\
                                                                                                                                        & = \int_a^b \int_{\RR^N}(-\Delta)^{\sigma/2}(u_\varepsilon^m-v_\varepsilon^m) f'_j(u_\varepsilon - v_\varepsilon)\widetilde{\xi}  - \varepsilon.
        \end{aligned}
    \end{equation}
    So far we have found that
    \begin{equation}
        \int_a^b \int_{\RR^N} f_j(u_\varepsilon - v_\varepsilon) \partial_t \widetilde{\xi}  + 2\varepsilon \geq \int_a^b \int_{\RR^N}(-\Delta)^{\sigma/2}(u_\varepsilon^m-v_\varepsilon^m) f'_j(u_\varepsilon - v_\varepsilon)\widetilde{\xi} .
    \end{equation}
    We now take the limit $j\to \infty$, and then use Kato's inequality~\eqref{eq:positivekato} to obtain
    \begin{equation}
        \begin{aligned}
            \int_a^b \int_{\RR^N} (u_\varepsilon - v_\varepsilon)_+ \partial_t \widetilde{\xi}  + 2\varepsilon & \geq \int_a^b \int_{\RR^N}(-\Delta)^{\sigma/2}(u_\varepsilon^m-v_\varepsilon^m) \chi_{ \{u_\varepsilon - v_\varepsilon \geq 0 \}}\widetilde{\xi}  \\
                                                                                                                      & \geq \int_a^b \int_{\RR^N}(-\Delta)^{\sigma/2}(u_\varepsilon^m-v_\varepsilon^m)_+ \, \widetilde{\xi}                                              \\
                                                                                                                      & = \int_a^b \int_{\RR^N}(u_\varepsilon^m-v_\varepsilon^m)_+ \, (-\Delta)^{\sigma/2}\widetilde{\xi} .
        \end{aligned}
    \end{equation}
    Letting $\varepsilon \to 0$ we find
    \begin{equation}
        \int_a^b \int_{\RR^N} (u - v)_+ \partial_t \widetilde{\xi}  \geq \int_a^b \int_{\RR^N}(u^m-v^m)_+ \, (-\Delta)^{\sigma/2}\widetilde{\xi} .
    \end{equation}
    Finally, choosing $\widetilde{\xi}(x,t)$ as an approximation of $\chi_{[t,t+h]}(t) \vartheta_R(x)$ and taking $h\to 0$, we obtain that 
    \begin{equation} \label{eq:almostfinished}
        \frac{d}{dt} \int_{\RR^N} (u - v)_+ \vartheta_R  \leq \int_{\RR^N}(u^m-v^m)_+ \, (-\Delta)^{\sigma/2}\vartheta_R .
    \end{equation}
    From this point onwards, the proof basically follows the formal one explained previously.
\end{proof}

\begin{Remark} \label{re:continuityremark}
    \emph{(i)} From expression~\eqref{eq:almostfinished} and estimate~\eqref{eq:comprecurrente} one can prove that weighted solutions satisfy that ${u \in \cont((0,T) \!: \! L^1_\vartheta) }$. In addition, if we use~\eqref{eq:weightedcomparison}, one can check that ${u \in \text{Lip}((0,T) \! : \!L^1_\vartheta)}$.

    \noindent\emph{(ii)} Using the previous result, we can now establish the existence of sign-changing solutions for~${0<m<1}$ starting from initial data a function $u_0$ satisfying $\displaystyle \int_{\RR^N} |u_0| \vartheta  < \infty$. This result was previously unattainable, since the estimate in~\Cref{re:orderedcomparison} requires ordered solutions. It can be proven through an approximation argument, just like in Theorem~3.1 of~\cite{BonforteVazquezQuantitativeEstimates}. 
\end{Remark}

\section{Existence of initial traces} \label{se:traces}

The following proof adapts~\cite[Theorem 7.2]{BonforteVazquezQuantitativeEstimates}, with minor modifications to account for weighted solutions.
\begin{Theorem} \label{re:initialtraces}
    Let $u$ be a nonnegative weighted solution of equation~\eqref{equation} in $\RR^N \times (0,T]$. Then, there exists a unique nonnegative Radon measure $\mu$ as initial trace, that is,
    \begin{equation}
        \lim_{t\to 0^+} \int_{\RR^N} u(x,t) \phi(x) \, dx = \int_{\RR^N} \phi(x) \, d\mu
    \end{equation}
    for all $\phi \in \cont_c(\RR^N)$. Moreover, the initial trace $\mu$ satisfies
    \begin{equation} \label{eq:muweighted}
        \lim_{t\to 0^+} \int_{\RR^N} u(x,t) \vartheta(x) \, dx = \int_{\RR^N} \vartheta(x) \, d\mu < \infty.
    \end{equation}
\end{Theorem}
\begin{proof}
    From~\Cref{re:orderedcomparison} and choosing $v=0$, we obtain
    \begin{equation}
        \int_{B_R} u(x,t) \, dx\leq \left(\left( \int_{\RR^N} u(x,T) \vartheta_R(x) \, dx\right)^{1-m} + \frac{C |T-t| }{R^{\sigma - N(1-m)}}\right)^{\frac{1}{1-m}}, 
    \end{equation}
    for $0 \leq t \leq T$. Choose any sequence $\{t_n\}_n \subset [0, \infty)$ such that $\lim_n t_n = 0$. Defining $\mu_n = u(t_n)\, dx$, we obtain through~\Cref{re:measurecompactness} the existence of a subsequence $\{t_{n_k}\}_k$ and a Radon measure $\mu$ such that
    \begin{equation}
        \lim_{k \to \infty} \int_{\RR^N} u(x,t_{n_k}) \phi(x) \, dx = \int_{\RR^N} \phi(x) \, d\mu.
    \end{equation}
    We now check that this measure $\mu$ does not depend on the sequence. From~\eqref{eq:salto_tiempo}, we get 
    \begin{equation} \label{eq:auxiliarintermedio}
        \left|\int_{\RR^N} u(\tau)\phi - \int_{\RR^N}u(t_{n_k}) \phi\right| = \int_{t_{n_k}}^{\tau} \int_{\RR^N} u(s)^m |(-\Delta)^{\sigma/2}\phi| \, ds, \leq C_\phi \int_{t_{n_k}}^\tau \|u(s)\|_{L^1_\vartheta(\RR^N)} \, ds.
    \end{equation}
    Taking $k\to \infty$ first, we see that 
    \begin{equation}
        \lim_{\tau\to 0^+} \int_{\RR^N} u(x,\tau) \phi(x) \, dx = \int_{\RR^N} \phi(x) \, d\mu
    \end{equation}
    for any $\phi \in \cont^\infty_c(\RR^N)$. The limit actually holds for any function in $\cont_c(\RR^N)$ through a mollification argument, see~\Cref{re:acotacionhastacero}. Now, in order to obtain~\eqref{eq:muweighted}, we just use $\vartheta$ instead of $\phi$ in~\eqref{eq:auxiliarintermedio}.
\end{proof}

With this result we now have the appropriate candidates for initial conditions.

\begin{Definition}
    We define $\mathfrak{M}_\sigma$ as the set of nonnegative Radon measures $\mu$ satisfying 
    \begin{equation}
        \int_{\RR^N} \vartheta \, d\mu < \infty
    \end{equation}
    for some $\vartheta \in \varTheta$. 
\end{Definition}

\section{Existence of solutions with initial data a measure} \label{se:existence}

We begin by recalling a result stated in~\cite{VazquezBarenblattnolocal} for the existence of weak solutions arising from a finite measure as initial condition.

\begin{Theorem} \label{re:existenciafinita}
    Let $m>m_c$ and let $\mu$ be a nonnegative finite Radon measure. Then there exists a weak solution $u$ to~\eqref{equation} taking $\mu$ as initial data in the vague topology, which conserves mass, 
    \begin{equation}
        \int_{\RR^N} u(x,t) \, dx = \mu(\RR^N) \quad \text{ for all } t>0.
    \end{equation}
    The solution satisfies the smoothing effect
    \begin{equation}
        \|u(t)\|_\infty \leq C t^{-\alpha} \mu(\RR^N)^{\gamma} \quad \text{ for all } t>0,
    \end{equation}
    with $\alpha = N /(N(m-1)+\sigma)$ and $\gamma = \sigma \alpha / N$, where the constant $C$ depends only on $m, N$, and $\sigma$.
\end{Theorem}

We now prove the following auxiliary result, which is fundamental both for the existence and uniqueness results, since it will allow us to define monotone sequences of weak solutions converging to some weighted solution. 

\begin{Lemma} \label{re:measurecomparison}
    Let $m_c < m < 1$, $\mu \in \mathfrak{M}_\sigma$, and let $u$ be a weighted solution of~\eqref{equation} with initial trace $\mu$. Then, for $h_1, h_2 \in \cont_c(\RR^N)$ with $0 \leq h_1 \leq h_2 \leq 1$, given a weak solution $w_1$ with initial condition $h_1 \mu$ there exists a weak solution $w_2$ with initial condition $h_2 \mu$ such that $w_1 \leq w_2 \leq u$ almost everywhere. 
\end{Lemma}
\begin{proof}
    Let us prove that for $h \in \cont_c(\RR^N)$ with $0 \leq h \leq 1$ there exists a solution $w$ such that $w \leq u$ almost everywhere. Proving the existence of the bound from below is completely analogous. 
    
    Let $w_n$ be the weak energy solution of~\eqref{equation} with initial condition $hu(1/n) \in L^1(\RR^N)$. It is obvious that $hu(1/n) \leq u(1/n)$. Then, thanks to the comparison principle of~\Cref{re:weightedcomparison}, we get that~$w_n(x,t) \leq u(x, t+ 1/n)$. First, observe that
    \begin{equation} \label{eq:integrallimit1}
        \lim_{n \to \infty} \int_{\RR^N} hu(1/n) \,dx = \int_{\RR^N} h \,d\mu,
    \end{equation}
    using $h$ as a test function and the fact that $u$ has $\mu$ as initial data.

    Now, let us show that $w_n$ converges, up to a subsequence, to a weak solution $w$ with initial condition $h\mu$. Observe that, thanks to the smoothing effect~\eqref{eq:smoothingeq} and the previous argument, we get equiboundedness of $w_n$,
    \begin{equation}
        \|w_n(t)\|_\infty \leq C t^{-\alpha}\|hu(1/n)\|_{1}^{\gamma} \leq Ct^{-\alpha} \quad \text{ for all } t>0.
    \end{equation}
    With this and the positivity result from~\cite{PositivityAndAsymptoticBehaviour}, we get that these solutions have the same Hölder constant and exponent, at least on compact sets, see~\cite{dePabloQuiros2018}. This implies that they are equicontinuous on those sets. Thus, using Arzelà-Ascoli’s Theorem and a diagonal argument we get that there is a subsequence converging on $\RR^N$ to some $w$. It is easy to see that this $w$ is a very weak solution of~\eqref{equation}. We have to check that $w$ is a weak energy solution with initial condition $h\mu$. We first prove that it belongs to the proper energy space, that is, $w^m \in L^2_{\text{loc}}((0,\infty); \dot{H}^{\sigma/2})$. Solutions $w_n$ verify the inequality
    \begin{equation}
        \int_{\frac{1}{T}}^T |(-\Delta)^{\sigma/4}w_n^m|^2 \, dt \leq \left\|\left(w_n\left( 1/T \right) \right) \right\|_{\infty}^m \left\|w_n\left( 1/T \right) \right\|_{1} \leq C T^{-m\alpha} \|h u(1/n)\|_1^{m+1} \leq C T^{-m\alpha},
    \end{equation}
    for any $T>1$. This previous estimate, at least formally, can be quickly proved using $w_n$ as a test function; while the rigorous argument requires the standard Steklov averages. The right-hand side does not depend on $n$, which can be seen taking advantage of the $L^1$ boundedness of $hu(1/n)$ we have been using so far. Then, thanks to Fatou's Lemma, we conclude that $w^m$ belongs to the energy space. Hence, $w$ is a weak solution.

    To check the initial condition of $w$, we use again the expression~\eqref{eq:salto_tiempo} for $w_n$, obtaining
    \begin{equation}
        \begin{aligned}
            \left| \int_{\RR^N} (w_n(t) - hu(1/n)) \phi  \right| & \leq  \int_0^t \int_{\RR^N} w_n^m \left|(-\Delta)^{\sigma/2}\phi \right| \leq t \|w_n\|_1^m \|(-\Delta)^{\sigma/2}\phi\|_{\frac{1}{1-m}} \\
                                                                & \leq t \|hu(1/n)\|_1^m \|(-\Delta)^{\sigma/2}\phi\|_{\frac{1}{1-m}}.
        \end{aligned}
    \end{equation}
    Therefore, taking first the limit $n\to \infty$ and then $t\to 0$ we find
    \begin{equation}
        \lim_{t\to 0} \int_{\RR^N} w(t) \phi = \int_{\RR^N}h \, d\mu.
    \end{equation}
    On the other hand, since we had $w_n(x,t) \leq u(x, t+ 1/n)$ we get that $w(x,t) \leq u(x,t)$ almost everywhere. 
\end{proof}

Once we have the previous result, the following proof, which is the goal of this section, is a straightforward adaptation of~\cite[Theorem 3.1]{BonforteVazquezQuantitativeEstimates}.

\begin{Theorem} \label{re:existencia}
    Let $m_c<m<1$ and $\mu \in \mathfrak{M}_\sigma$. Then there exists ${u \in \cont((0,\infty) \! : \! L^1_\vartheta)}$ a nonnegative weighted solution to equation~\eqref{equation} with $\mu$ as its initial condition.
\end{Theorem}
\begin{proof}
    Choose a sequence of smooth compactly supported functions $\{h_n\}_n$ such that $h_n \leq h_{n+1}$ and~$\lim_{n\to \infty} h_n = 1$. Let $u_1$ be the solution with initial condition $h_1 \mu$, which exists due to~\Cref{re:existenciafinita}. Now, define $u_{2}$ as the weak solution with initial data $h_{2}\mu$ such that $u_2 \geq u_1$, which exists thanks to~\Cref{re:measurecomparison}. Iterating this process, we have a monotone sequence $\{u_n\}_n$ of weak solutions. Now, using~\Cref{re:orderedcomparison}, we obtain that 
    \begin{equation}
        \left( \int_{\RR^N} u_n(x,t) \vartheta_R(x) \, dx\right)^{1-m} \leq \left( \int_{\RR^N} u_n(x,\tau ) \vartheta_R(x) \, dx\right)^{1-m} + \frac{C |t-\tau| }{R^{\sigma - N(1-m)}}.
    \end{equation}
    Observe that taking limits on the right-hand side, first $\tau \to 0$ and then $n \to \infty$, yields
    \begin{equation}
        \left( \int_{\RR^N}  \vartheta_R \, d\mu \right)^{1-m} + \frac{C t }{R^{\sigma - N(1-m)}}.
    \end{equation}
    Thus, using monotonicity we can take the limit in the previous expression to deduce that $u_n$ converges to a function ${u \in L^\infty_{\text{loc}}((0,\infty) \! : \! L^1_\vartheta)}$ satisfying
    \begin{equation}
        \left( \int_{\RR^N} u(x,t) \vartheta_R(x) \, dx\right)^{1-m} \leq \left( \int_{\RR^N} \vartheta_R \, d\mu \right)^{1-m} + \frac{C t }{R^{\sigma - N(1-m)}}.
    \end{equation}
    It is standard to check that $u$ is a weighted solution. It remains to see
     that $u$ attains $\mu$ as its initial condition. From~\eqref{eq:salto_tiempo}, we get
    \begin{equation}
        \left|\int_{\RR^N} u_n(t)\phi- \int_{\RR^N}h_n \phi \, d\mu \right|  \leq \int_0^{t} \int_{\RR^N} u_n^m(s) |(-\Delta)^{\sigma/2}\phi| \, ds \leq \int_0^{t} \int_{\RR^N} u^m(s) \vartheta^m \frac{|(-\Delta)^{\sigma/2}\phi|}{\vartheta^m} \, ds. 
    \end{equation} 
    Since ${u \in \cont((0,\infty) \! : \! L^1_\vartheta)}$, we use~\eqref{eq:comprecurrente} again and take the limit $n \to \infty$ first and then $t \to 0$ to conclude 
    \begin{equation} \tag*{ }
        \lim_{t\to 0^+} \int_{\RR^N} u(x,t) \phi(x) \, dx = \int_{\RR^N} \phi \, d\mu. \qedhere
    \end{equation} 
\end{proof}

\section{Uniqueness} \label{se:uniqueness}

This section requires several lemmas before establishing the main results. 

\subsection{Auxiliary results}

Let us start by showing that the potential $U = (\mathcal{I}_\sigma * u)$ is nonincreasing in time. 

\begin{Lemma} \label{re:monotoniapotencial}
    Let $u$ be a nonnegative weak solution of~\eqref{equation} and $\mathcal{I}_\sigma$ be the Riesz potential. If we denote $U(x,t) = (\mathcal{I}_\sigma * u)(x,t)$ the potential of $u$, then, for every $x \in \RR^N$, $U(x,t)$ is nonincreasing as $t$ grows, 
    \begin{equation}
        U(x,t_1) \geq U(x,t_2), \qquad 0 \leq t_1 \leq t_2.
    \end{equation}
\end{Lemma}
\begin{proof}
    We start from a stronger version of expression~\eqref{eq:salto_tiempo}, available for weak solutions,
    \begin{equation} \label{eq:weakjumpintime}
        \int_{\RR^N}u(t) \phi - \int_{\RR^N}u(s) \phi = - \int_s^t \int_{\RR^N} (-\Delta)^{\sigma/4}u^m(\tau) (-\Delta)^{\sigma/4}\phi \, d\tau, 
    \end{equation}
    with $\phi \in \cont^\infty_c(\RR^N)$. We would like to take $\widetilde{\phi} = \mathcal{I}_\sigma * \phi$ as a test, where $\phi$ is a compactly supported smooth function. Thus, we want to prove that $\widetilde{\phi} \in \dot{H}^{\sigma/2}(\RR^N)$, which requires checking
    \begin{equation}
        \begin{aligned}
            \widetilde{\phi} \in L^{\frac{2N}{N-\sigma}}(\RR^N) \quad \text{and} \quad (-\Delta)^{\sigma/4}\widetilde{\phi} \in L^2(\RR^N).
        \end{aligned}
    \end{equation}
    Since $\mathcal{I}_\sigma \in L^{\frac{N}{N-\sigma}, \infty}$, it follows from standard convolution theory that
    \begin{equation}
        \widetilde{\phi} = \mathcal{I}_\sigma * \phi \in L^p(\RR^N) \cap C^\infty(\RR^N) 
    \end{equation}
    for all $p \in \left(\frac{N}{N-\sigma}, \infty \right]$. This interval contains the desired $L^{\frac{2N}{N-\sigma}}(\RR^N)$ space, leaving only the energy condition to be checked. It can be seen through the use of the Fourier transform that
    \begin{equation}
        (-\Delta)^{\sigma/4}\widetilde{\phi} = \mathcal{I}_{\sigma/2} *\phi,
    \end{equation}
    and, using the same convolution arguments this last function belongs to $L^2(\RR^N)$ if $2>\frac{N}{N-\sigma/2}$. This last condition is equivalent to $N>\sigma$, which we are assuming.

    Now, using $\widetilde{\phi}$ as a test function in~\eqref{eq:weakjumpintime} we get
    \begin{equation}
        \int_{\RR^N}u(t) \big(\mathcal{I}_{\sigma} *\phi\big) - \int_{\RR^N}u(s) \big(\mathcal{I}_{\sigma} * \phi \big) = - \int_s^t \int_{\RR^N} (-\Delta)^{\sigma/4}u^m(\tau) \, \big( \mathcal{I}_{\sigma/2}*\phi \big) \, d\tau 
    \end{equation}
    or, equivalently,
    \begin{equation}
        \int_{\RR^N}U(t) \phi - \int_{\RR^N}U(s) \phi = - \int_s^t \int_{\RR^N}u^m(\tau) \phi \, d\tau. 
    \end{equation}

    Let us justify the computation of the right-hand side term. Since $u^m \in L^2_{\text{loc}}((0,\infty)\! : \!\dot{H}^{\sigma/2})$, we can find a sequence $\zeta_n \in \cont^\infty_c(\RR^N\times (0,T))$ such that $\zeta_n \rightarrow u^m$ in $L^2((0,T)\! : \!\dot{H}^{\sigma/2})$ for $T>t\geq s>0$. Notice that this also implies convergence in $L^2((0,T)\! : \!L^{\frac{2N}{N-\sigma}})$. It is easy to see, thanks to the Fourier transform, that $(-\Delta)^{\sigma/4}\xi_n(t) \in \mathcal{S}(\RR^N)$. Then,
    \begin{equation}
        \int_{\RR^N} (-\Delta)^{\sigma/2} \zeta_n(t) \, \big(\mathcal{I}_{\sigma/2}*\phi \big) = \int_{\RR^N} |\xi|^{\sigma/2} \widehat{\zeta_n}(t) |\xi|^{-\sigma/2} \widehat{\phi}  = \int_{\RR^N} \zeta_n(t) \phi.  
    \end{equation}

    Since $\int_s^t u^m \geq 0$, choosing positive test functions $\phi$ we get that if $t\geq s>0$, $U(x,t)\leq U(x,s)$ for a.e.\ $x \in \RR^N$. But, for every $t$, the potential $U(\cdot, t)$ is continuous in $x$ thanks to convolution, so the monotonicity can be extended to every $x$ in $\RR^N$.
\end{proof}

The next result is a very simple lemma to easily check when a function $f$ is going to be C-absolutely continuous, as defined in~\Cref{se:preliminaries}. 

\begin{Lemma} \label{re:cabsolutelyaluso}
    Let $f \in L^1 \cap L^\infty(\RR^N)$. Then, $f\, dx$ is C-absolutely continuous.
\end{Lemma}
\begin{proof}
    Due to~\Cref{re:criteriocabsolutely}, we just need to check that $U^f_\sigma \in \cont_b(\RR^N)$. 
    
    Recall that $\mathcal{I}_\sigma = C_{N, \sigma} |x|^{-N+\sigma}$, hence $\mathcal{I}_\sigma \chi_{B_1} \in L^{q_1}(\RR^N)$ for $q_1 < N/(N-\sigma)$ and $\mathcal{I}_\sigma \chi_{B^c_1} \in L^{q_2}(\RR^N)$ for $q_2 > N/(N-\sigma)$. Then, using Hölder's inequality, we obtain
    \begin{equation}
        \|\mathcal{I}_\sigma * f\|_\infty \leq \|\mathcal{I}_\sigma \chi_{B_1} * f\|_\infty + \|\mathcal{I}_\sigma \chi_{B^c_1} * f\|_\infty \leq \|f\|_{p_1} \|\mathcal{I}_\sigma \chi_{B_1}\|_{q_1} + \|f\|_{p_2} \|\mathcal{I}_\sigma \chi_{B_1}\|_{q_2}
    \end{equation}
    for some $p_1 > N/\sigma$ and $1< p_2 < N/\sigma$, since $N>\sigma$. Continuity follows from the following standard argument: if $\tau_h$ is the translation operator of distance $h$, then
    \begin{equation}
        \|\tau_h U^f_\sigma - U^f_\sigma\|_\infty \leq \|\tau_h f - f\|_{p_1} \|\mathcal{I}_\sigma \chi_{B_1}\|_{q_1} + \|\tau_h f - f\|_{p_2} \|\mathcal{I}_\sigma \chi_{B_1}\|_{q_2} \to 0
    \end{equation}
    as $h \to 0$.
\end{proof}

The next result is fundamental, since the uniqueness proof formally relies on using the function $\psi$ defined below as a test function for our problem.
\begin{Lemma} \label{re:auxequations}
    Let $\lambda_1, \lambda_2>0$ and $\alpha \in \cont^\infty(\RR^N)$ such that $\lambda_1 \leq \alpha \leq \lambda_2$. Then, there exists $h, \psi: \RR^N \times (a,b) \to \RR$ classical solutions to the backwards problems
    \begin{equation}
        \label{heq}
        \begin{cases}
            \partial_t h - (-\Delta)^{\sigma/2}(\alpha  h) = 0, & \RR^N\times(a,b), \\
            h(x,b) = \eta(x),  \qquad &x \in \RR^N,
        \end{cases}
    \end{equation}
    \begin{equation}
        \label{psieq}
        \begin{cases}
            \partial_t \psi -\alpha (-\Delta)^{\sigma/2} \psi = 0,  & \RR^N\times(a,b), \\
            \psi(x,b) = \theta(x),  \qquad &x \in  \RR^N,
        \end{cases}
    \end{equation}
    with $\eta \in \cont^\infty_c(\RR^N)$ and $\theta = \mathcal{I}_\sigma * \eta$. They satisfy
    \begin{align}
        \label{mass}
        \int_{\RR^N} h(x,t) \, dx                  & = \int_{\RR^N} \eta(x) \, dx,                                                                              \\
        \label{coeficientecuadrado}
        \iint_{Q_a^b} \alpha |(-\Delta)^{\sigma/2}\psi|^2 & \leq \frac{1}{2}\int_{\RR^N} |(-\Delta)^{\sigma/4} \theta|^2 = \frac{1}{2}\int_{\RR^N} \eta \theta, \\
        \label{cuadrado_unif_n}
        \sup_{a \leq t \leq b}\int_{\RR^N} |(-\Delta)^{\sigma/4}\psi
        |^2                                               & \leq \int_{\RR^N} \eta \theta.
    \end{align}
    Solution $\psi$ verifies
    \begin{equation}
        \label{monotonicity}
        0 \leq \psi(t_1) \leq \psi(t_2) \leq \psi(b) = \theta, \quad \text{for any } a\leq t_1 \leq t_2 \leq b,
    \end{equation}
    so it is bounded in $L^p(\RR^N)$ for $p \in (N/(N-\sigma),\infty]$ uniformly in $t$.
\end{Lemma}
\begin{proof}
    To find a solution for the equation~\eqref{heq}, we first solve
    \begin{equation}
        \begin{cases}
            \partial_t \zeta - \alpha(-\Delta)^{\sigma/2}\zeta = 0,  & \RR^N\times(a,b), \\
            \zeta(x,b) = \alpha(x)\eta(x),  \qquad &x \in \RR^N.
        \end{cases}
    \end{equation}
    This can be done using the method described in~\cite{ChenZhang2016_Parametrix} and~\cite{ChenZhang2016_ParametrixWithTime}, or in the appendix found in~\cite{dePabloQuiros2016}, which is based on the parametrix method by E. Levi. We get the solution $h$ by simply defining it as $h \coloneq \zeta/\alpha$. This function $h$ inherits its regularity properties from $\zeta$, and additionally it verifies conservation of mass~\eqref{mass}. Let us define $\psi \coloneq \mathcal{I}_\sigma * h$, so that $h=(-\Delta)^{\sigma/2} \psi$, or
    \begin{equation}
        \mathcal{I}_\sigma * (-\Delta)^{\sigma/2} \psi = \mathcal{I}_\sigma * h = \psi. 
    \end{equation}
    This $\psi$ solves equation~\eqref{psieq}.     Multiplying equation~\eqref{psieq} by $(-\Delta)^{\sigma/2}\psi_n$ and integrating in $Q_t^b = \RR^N \times (t,b)$,
    \begin{equation}
        \iint_{Q_t^b} \alpha |(-\Delta)^{\sigma/2}\psi|^2 + \frac{1}{2}\int_{\RR^N} |(-\Delta)^{\sigma/4} \psi(t)|^2 = \frac{1}{2}\int_{\RR^N} |(-\Delta)^{\sigma/4} \theta|^2 = \frac{1}{2}\int_{\RR^N} \eta \theta.
    \end{equation}
    From this inequality, by taking the supremum for $t \in (a,b)$, we obtain both~\eqref{coeficientecuadrado} and~\eqref{cuadrado_unif_n}. Since $h,\alpha \geq 0$, and $\partial_t \psi = \alpha h$, the function $\psi$ is increasing in time, and we get
    \begin{equation}
        0 \leq \psi(t_1) \leq \psi(t_2) \leq \psi(b) = \theta, \quad \text{for any }0\leq t_1 \leq t_2 \leq T.
    \end{equation}
    Finally, observe that the function $\theta$ is the convolution of $\mathcal{I}_\sigma \in L^{N/(N-\sigma), \infty}(\RR^N)$ and $\eta \in \cont^\infty_c(\RR^N)$, so we find $\theta \in \cont^\infty(\RR^N) \cap L^p(\RR^N)$ for $p \in (N/(N-\sigma),\infty]$.
\end{proof}

The next auxiliary result will be helpful to avoid using a coefficient depending on time for the backwards problem~\eqref{psieq} that $\psi$ satisfies. 

\begin{Lemma} \label{re:approximationlemma}
    Let $R, T>0$ and $f: \RR^N \times (a,b] \to \RR$ be a measurable function such that $0 \leq f \leq 1$ almost everywhere. Then, there exists $f_n: \RR^N \times (a,b] \to \RR$ defined piecewise in time in each subinterval $\left(a+\frac{k(b-a)}{m}, a+\frac{(k+1)(b-a)}{m} \right]$ for $k=0, \dots, m-1$ such that $f(x,t)$ is smooth in $x$ for each fixed $t \in (a, b]$ and satisfies
    \begin{equation}
        \frac{1}{n} \leq f_n\leq 2, \quad \|f-f_n\|_{L^2(B_R\times(a,b))} < \frac{1+ 2 \sqrt{(b-a) |B_R|}}{n}.
    \end{equation}
\end{Lemma}
\begin{proof}
    Let $g_n \in \cont^\infty(\RR^{N+1})$ be such that $\|f-g_n\|_{L^2(B_R\times(a,b))} < \frac{1}{n}$ and $0\leq g_n \leq 1$. Now, divide the time interval $(a, b]$ in subintervals
    \begin{equation}
        \left(a+\frac{k(b-a)}{m}, a+\frac{(k+1)(b-a)}{m} \right]
    \end{equation}
    for $m \in \mathbb{N}$ and $k = 0, \dots, m-1$. Define the function $g^m_n$ as
    \begin{equation}
        g^m_n(x,t) = g_n\left(x, a+\frac{(k+1)(b-a)}{m} \right)
    \end{equation}
    for $t \in \left(a+\frac{k(b-a)}{m}, a+\frac{(k+1)(b-a)}{m} \right]$ and each $k = 0, \dots, m-1$. Since $g_n$ is uniformly continuous in $B_R \times [a,b]$, there exists $m_n$ big enough depending on $n$ so that these functions verify
    \begin{equation}
        \left|g_n (x,t) - g_n \left(x, a+\frac{(k+1)(b-a)}{m_n} \right) \right| \leq \frac{1}{n}
    \end{equation}
    for $t \in \left(a+\frac{k(b-a)}{m}, a+\frac{(k+1)(b-a)}{m} \right]$. Hence,
    \begin{equation}
        \|g_n - g^{m_n}_n\|_{L^2(B_R \times (a, b))} \leq \frac{\sqrt{(b-a) |B_R|}}{n}.
    \end{equation}
    Lastly, it is obvious that
    \begin{equation}
        \left\| \frac{1}{n} \right\|_{L^2(B_R \times (a, b))} = \frac{\sqrt{(b-a) |B_R|}}{n}.
    \end{equation}
    Therefore, defining $f_n$ as
    \begin{equation}
        f_n(x,t) \coloneq g^{m_n}_n(x,t) + \frac{1}{n},
    \end{equation}
    we get that $1/n \leq f_n(x,t) \leq 2$ and
    \begin{equation+} \tag*{ }
        \|f-f_n\|_{L^2(B_R\times(a, b))} \leq \frac{1+ 2 \sqrt{(b-a) |B_R|}}{n}. \qedhere
    \end{equation+} 
\end{proof}

\subsection{Main results}

We start this subsection with the longest (and most technical) result of this work which states that weak solutions to equation~\eqref{equation} starting from a finite Radon measure as an initial condition are unique. This is already a new result for the fractional fast diffusion equation, analogue to the one for the fractional slow diffusion equation stated in~\cite{GrilloMuratoriPunzo2015_Measure}. In fact, with only minor changes, the following result is also valid for $m \geq 1$, therefore proving uniqueness of weak solutions for any $m>m_c$, see~\Cref{re:slowdifchanges}.

Let us explain the core idea behind the proof. We subtract the expressions~\eqref{eq:veryweakwithinitialdata} for $u$ and $v$, and then rearrange the identity in order to get 
\begin{equation}
    \begin{gathered}
        \int_{\RR^N} (u(x,T)-v(x,T)) \xi(x,T) dx -  \int_{\RR^N} (u(x,\tau)-v(x,\tau)) \xi(x,\tau) dx  \\
        = \iint_{Q_T^\tau} (u-v) \partial_t \xi\, dx dt - (u^m-v^m) (-\Delta)^{\sigma/2} \xi \, dx dt \\
        = \iint_{Q_T^\tau} (u-v) \left\{\partial_t \xi\, dx dt - \left(\frac{u^m-v^m}{u-v}\right) (-\Delta)^{\sigma/2} \xi \right\} \, dx dt.
    \end{gathered}
\end{equation}
Now observe that if we chose $\xi$ as the solution $\psi$ of the backwards problem~\eqref{psieq} with $\alpha = \frac{u^m-v^m}{u-v}$ and $(a,b) = (\tau, T)$, we would obtain that
\begin{equation}
    \int_{\RR^N} (u(x,T)-v(x,T)) \theta(x) \, dx -  \int_{\RR^N} (u(x,\tau)-v(x,\tau)) \psi(x,\tau) \, dx = 0.
\end{equation}
Formally, taking the limit $\tau \to 0$ would finish the argument. To justify all the previous steps we first need to deal with the coefficient $\frac{u^m - v^m}{u-v}$. In this fast diffusion equation, this coefficient can be both degenerate (approach $0$) and singular (approach $\infty$). This makes this case more complicated than the linear case, where there is no coefficient, and the slow diffusion case, where the coefficient can be degenerate but not singular. 

Recall that weak solutions satisfy conservation of mass, so they belong to ${L^\infty((0, \infty) \!:  \!L^1(\RR^N))}$, and the smoothing effect, therefore for any $\tau>0$ they belong to ${L^\infty([\tau,\infty) \!:  \! L^\infty(\RR^N))}$. 

\begin{Theorem} \label{re:firstuniqueness}
    Let $m>m_c$, $\mu$ be a finite Radon measure, and let $u,v$ be nonnegative weak solutions to~\eqref{equation} in $(0,T)$ such that
    \begin{equation}
        \lim_{t \to 0}\int_{\RR^N} u(t) \phi \, dx = \lim_{t \to 0}\int_{\RR^N} v(t) \phi \, dx = \int_{\RR^N} \phi \, d\mu
    \end{equation}
    for any $\phi \in \cont_c(\RR^N)$. Then, $u = v$ almost everywhere in $\RR^N \times (0, T)$.
\end{Theorem}
\begin{proof}
    Let us define $v_s(x,t) = v(x, t+s)$, which is also a solution. We define $\overline T$ so that $\overline{T} + s < T$, but we will simply write $T$. Let us subtract the expressions~\eqref{eq:veryweakwithinitialdata} for $u$ and $v$ so that we get
    \begin{equation} \label{eq:subtraction}
        \begin{aligned}
            \int_{\RR^N} (u(x,T)-v_s(x,T)) \xi(x,T) \, dx & -  \int_{\RR^N} (u(x,\tau)-v_s(x,\tau)) \xi(x,\tau) \,dx                 \\
                                                            & = \iint_{Q_T^\tau} (u-v_s) \partial_t \xi\, dx dt                             \\
                                                            & - \iint_{Q_T^\tau}(u^m-v_s^m) (-\Delta)^{\sigma/2} \xi \, dx dt
        \end{aligned}
    \end{equation}

    Following~\cite{CrowleyTrick}, we define $A$ and $B$ as
    \begin{equation}A=\left\{\begin{aligned}
             & \frac{u-v_s}{u-v_s+u^m-v_s^m} &  & \text{ when } u\neq v_s, \\
             & \quad\quad\quad\quad 0 \quad          &  & \text{ when } u=v_s,
        \end{aligned}\right. 	\quad  B=\left\{\begin{aligned}
             & \frac{u^m-v_s^m}{u-v_s+u^m-v_s^m} &  & \text{ when } u\neq v_s, \\
             & \quad\quad\quad\quad 0 \quad                            &  & \text{ when } u=v_s.
        \end{aligned}\right. 	 \end{equation}

    Both $A$ and $B$ satisfy $0\leq A,B\leq 1$. Observe that $\alpha = B/A$. We can assume $\alpha$ to be well-defined because if $A=0$, then $u=v_s$ and $u^m=v_s^m$, which already gives us what we want to get. So we ignore the set in which $A = 0$. Instead of approximating $\alpha$ directly, we approximate $A$ and $B$ like in~\cite{PerthameQuirosVazquez_HeleShaw}. Let us fix $R>0$. Then, we build $A_n$ and $B_n$ as in~\Cref{re:approximationlemma}, with the added restriction that they both must be piecewise constant in time in the same intervals. Therefore, we have
    \begin{equation}
        \left\{\begin{aligned}
             & \|A-A_n\|_{L^2(B_R\times(0,T))} < \beta/n, \quad  &  & \frac{1}{n}<A_n\leq 2, \\
             & \|B-B_n\|_{L^2(B_R\times(0,T))} < \beta/n,  \quad &  & \frac{1}{n}<B_n\leq 2, \\
        \end{aligned}\right.
    \end{equation}
    with $\beta$ depending on $R$ and $T$. Observe that $\alpha_n=B_n/A_n$ is also piecewise constant in time in the same time intervals as $A_n$ and $B_n$. It follows from these definitions that $\frac{1}{2n} \leq \alpha_n \leq 2n$.

    Let us build $h_n$ and $\psi_n$ as in~\Cref{re:auxequations} with coefficient $\alpha = \alpha_n$. We construct $h_n$ and $\psi_n$ for each subinterval in which $\alpha_n$ is constant in time, and then glue them together by choosing as the initial condition of $h_n$ in each subinterval the final time of the previous subinterval. All the estimates of $h_n$ and $\psi_n$ found in~\Cref{re:auxequations} are preserved.

    Consider $\xi = \phi_R \psi_n$ our test function, with $\phi_R$ a smooth function having compact support in $B_R$, $\phi_R \equiv 1$ if $|x|\leq R/2$ and $\|\phi_R\|_\infty\leq 1$. In fact, we consider it our test function for each time interval in which $\alpha_n$ is constant in time, but after joining all the intervals from~\eqref{eq:subtraction} we get
    \begin{equation} \label{limitenR}
        \begin{aligned}
            \int_{\RR^N} (u(x,T)-v_s(x,T)) \phi_R(x) \theta(x) \,dx & -  \int_{\RR^N} (u(x,\tau)-v_s(x,\tau)) \phi_R(x) \psi_n(x,\tau) \,dx \\
                                                                       & =  I_n^1(R) + I_n^2(R) - I_n^3(R) + I_n^4(R),
        \end{aligned}
    \end{equation}
    with
    \begin{align}
         & I_n^1(R) = \iint_{Q_\tau^T}(u-v_s+u^m-v_s^m)\frac{B_n}{A_n}(A-A_n)\phi_R(-\Delta)^{\sigma/2}\psi_n, \\
         & I_n^2(R) = \iint_{Q_\tau^T}(u-v_s+u^m-v_s^m)(B_n-B)\phi_R(-\Delta)^{\sigma/2}\psi_n,                \\
         & I_n^3(R) = \iint_{Q_T^\tau}  (u^m-v_s^m)  \psi_n (-\Delta)^{\sigma/2} \phi_R,                     \\
         & I_n^4(R) = \iint_{Q_T^\tau}  (u^m-v_s^m) E(\psi_n, \phi_R ),
    \end{align}
    where
    \begin{equation}
        E(f, g)(x) = \frac{1}{2}\int_{\RR^N} \frac{(f(x)-f(y))(g(x)-g(y))}{|x-y|^{N+\sigma}} \,dy.
    \end{equation}

    Observe that $I^1_n$ and $I^2_n$ appear due to the approximation of the coefficient $\alpha_n$, and $I^3_n$ and $I^4_n$ due to the cut-off function $\phi_R$. With the use of~\eqref{coeficientecuadrado} and~\eqref{cuadrado_unif_n}, the cut-off function $\phi_R$, and the boundedness of $u, v_s,u^m, v_s^m$ for a $\tau>0$ we can perform the following computations
    \begin{equation}
        \begin{aligned} | I_n^1 |
             & \leq C \iint_{B_R\times(\tau,T)} \frac{B_n}{A_n} |A-A_n| |(-\Delta)^{\sigma/2} \psi_n|                                                                                                                      \\
             & \leq C  \left\| \left(\frac{B_n}{A_n}\right)^{1/2}((-\Delta)^{\sigma/2} \psi_n) \right\|_{L^2(B_R\times(\tau,T))} \left\| \left( \frac{B_n}{A_n} \right)^{1/2} (A-A_n)    \right\|_{L^2(B_R\times(\tau,T))} \\
             & \leq  C n^{1/2} \left\|A-A_n    \right\|_{L^2(B_R\times(\tau,T))}  \leq C\alpha/n^{1/2},
        \end{aligned}
    \end{equation}
    \begin{equation}
        \begin{aligned} | I_n^2 |
             & \leq C \iint_{B_R\times(\tau,T)} |B-B_n| |(-\Delta)^{\sigma/2} \psi_n|                                                                                                                                      \\
             & \leq C  \left\| \left(\frac{B_n}{A_n}\right)^{1/2}((-\Delta)^{\sigma/2} \psi_n) \right\|_{L^2(B_R\times(\tau,T))} \left\| \left( \frac{A_n}{B_n} \right)^{1/2} (B-B_n)    \right\|_{L^2(B_R\times(\tau,T))} \\
             & \leq  C n^{1/2} \left\|B-B_n    \right\|_{L^2(B_R\times(\tau,T))}  \leq C\beta/n^{1/2}.
        \end{aligned}
    \end{equation}

    The constant $C$ may degenerate when $\tau \to 0$ and when $R \to \infty$. To deal with $I_n^3(R)$, we use similar arguments to those employed in the standard conservation of mass proofs
    \begin{equation}
        \begin{aligned}
            |I_n^3(R)|=\left|\iint_{Q_T^\tau}  (u^m-v_s^m) ( \psi_n \mathcal{L} \phi_R) \right| & \leq \|\theta\|_\infty \iint_{Q_T^\tau}  (|u^m| + |v_s^m|) |\mathcal{L} \phi_R | \\
                                                                                                            & \leq C \|\theta\|_\infty(T-\tau) \||u^m| + |v_s^m|\|_{1/m} R^{N(1-m)-\sigma}.
        \end{aligned}
    \end{equation}
    Here we use $m>(N-\sigma)/N = m_c$ so that this term goes to zero. To estimate $I_n^4(R)$ we need to use Hölder's inequality twice. Let us write $\Phi(x,t) = u^m(x,t)-v_s^m(x,t)$. We have
    \begin{equation}
        \begin{aligned}
            |I_n^4(R)| & \leq \left| \iint_{Q_T^\tau}  \Phi(x,t) E(\psi_n(x,t), \phi_R(x) ) \, dxdt \right|                                                                                              \\
                       & \leq \int\limits_\tau^T \int\limits_{\RR^N} |\Phi(x,t)| \left( \int\limits_{\RR^N} \frac{|\psi_n(x,t)-\psi_n(y,t)|\, |\phi_R(x)-\phi_R(y)|}{|x-y|^{N+\sigma}} \,dy \right) dx dt.
        \end{aligned}
    \end{equation}
    We use Hölder's inequality in the interior integral, so that we get
    \begin{equation}
        \begin{aligned}
             & \int\limits_{\RR^N} \frac{|\psi_n(x,t)-\psi_n(y,t)|\, |\phi_R(x)-\phi_R(y)|}{|x-y|^{N+\sigma}} \,dy = \int\limits_{\RR^N} \left|\frac{\psi_n(x,t)-\psi_n(y,t)}{|x-y|^{(N+\sigma)/2}} \right| \left|\frac{\phi_R(x)-\phi_R(y)}{|x-y|^{(N+\sigma)/2}}\right| \,dy \\
             & \leq \left( \int\limits_{\RR^N} \frac{(\psi_n(x,t)-\psi_n(y,t))^2}{|x-y|^{N+\sigma}}\, dy \right)^{1/2} \left(\int\limits_{\RR^N} \frac{(\phi_R(x)-\phi_R(y))^2}{|x-y|^{N+\sigma}}\, dy  \right)^{1/2}.
        \end{aligned}
    \end{equation}
    Define
    \begin{equation}
        \overline{E}(f)(x) =\int\limits_{\RR^N} \frac{(f(x)-f(y))^2}{|x-y|^{N+\sigma}}\, dy  .
    \end{equation}
    We have that $\overline{E}(\phi_R)(x) = R^{-\sigma} \overline{E}(\phi_1)(x/R)$. If we tried to exactly replicate the inequalities done for $I_n^3(R)$, we would encounter a problem here. This is apparent after taking the $L^q(\RR^N)$ norm,
    \begin{equation}
        \|(\overline{E}(\phi_R)(x))^{1/2}\|_q = \|\overline{E}(\phi_R)(x)\|_{q/2}^{1/2}  \leq C R^{-\sigma/2+N/q}.
    \end{equation}
    This estimate is worse than the one of $I_n^3(R)$ in the sense that we are going to require $q$ to be larger. If we want the power of $R$ to be negative, a possible choice would be $q=\frac{2}{1-m}$. Let us use Hölder's inequality with three terms, which requires
    \begin{equation}
        \frac{1}{q_1} + \frac{1}{q_2} + \frac{1}{q_3} = 1.  
    \end{equation}
    With this,
    \begin{equation}
        |I_n^4(R)| \leq \int\limits_\tau^T \|\Phi(\cdot,t)\|_{q_1} \|(\overline{E}(\psi_n)(x,t))^{1/2}\|_{q_2} \|(\overline{E}(\phi_R)(x))^{1/2}\|_{q_3}.
    \end{equation}
    Choose $q_1 = \frac{2}{m}$, $q_2 = 2$ and $q_3 =\frac{2}{1-m}$. Thanks to $q_2 = 2$, we can make use of~\eqref{cuadrado_unif_n} and thus get a bound uniform in $n$. Moreover, $q_1 > \frac{1}{m}$ and $\Phi \in L^{1/m}(\RR^N) \cap L^\infty(\RR^N)$ uniformly in time. Hence, we conclude that
    \begin{equation}
        |I_n^4(R)| \leq C R^{\frac{-\sigma + N(1-m)}{2}} \underset{R\rightarrow \infty}{\longrightarrow} 0,\quad \text{uniformly in $n$},
    \end{equation}
    since the constant $C$ does not depend on $n$ thanks to~\eqref{cuadrado_unif_n} and, again, using that $m>m_c$.

    Thanks to these bounds, we can control the right-hand side of~\eqref{limitenR}: integrals $I_n^1(R)$, $I_n^2(R)$ go to zero when $n \to \infty$, and $I_n^3(R)$, $I_n^4(R)$ are bounded uniformly in $n$ and go to zero when we take the limit $R \to \infty$. But we have not controlled yet how the left-hand side of~\eqref{limitenR} behaves. The integral
    \begin{equation}
        \int_{\RR^N} (u(x,T)-v_s(x,T)) \phi_R(x) \theta(x) \,dx 
    \end{equation}
    is not going to be a problem, but we need to deal with
    \begin{equation}
        \int_{\RR^N} (u(x,\tau)-v_s(x,\tau)) \phi_R(x) \psi_n(x,\tau) \, dx 
    \end{equation}
    when $n$ goes to infinity. We are going to extract an adequate converging subsequence of $\psi_n$.

    First, let us choose a sequence $\tau_j$ converging to zero as $j\rightarrow \infty$. The potentials $\psi_n(\tau_j)$, are uniformly bounded in $L^p(\RR^N)$ for a fixed $p\in \left(N/(N-\sigma), \infty \right]$ since $0\leq \psi_n(\tau_j) \leq \theta$. This allows us to choose a subsequence in $n$ such that $\psi_n(\tau_j)$ converges weakly to $\Psi^j$. For each $j$, we may find a different subsequence, but we can get rid of this problem after extracting another subsequence valid for all $j$ using a diagonal argument. For any limit that might depend on $j$ we perform the previous argument to avoid it. We would like to identify this limit $\Psi^j$ as another potential. Thanks to the Banach-Saks Theorem,~\Cref{re:BanachSaks}, there exists a subsequence $\{\psi_{n_k}(\tau_j)\}_k$ of the previous one such that
    \begin{equation} \label{eq:cesaromean}
        \Psi_N(\tau_j) = \frac{1}{N} \sum_{k=1}^N \psi_{n_k}(\tau_j)
    \end{equation}
    converges $\lim_{N \to \infty} \Psi_N(\tau_j) = \Psi^j$ strongly in $L^p$. Extracting a subsequence in $N$ of $\Psi_N(\tau_j)$ we conclude that $\Psi_N(\tau_j)$ converges a.e.\ to $\Psi^j$. Now define
    \begin{equation}
        H_N(\tau_j) = \frac{1}{N} \sum_{k=1}^N h_{n_k}(\tau_j). 
    \end{equation}
    It is easy to see that $\Psi_N(\tau_j) = \mathcal{I}_\sigma*H_N(\tau_j)$ due to linearity. The functions $H_N(\tau_j)$ conserve the same mass as $h_n(\tau_j)$, 
    \begin{equation}
        \int_{\RR^N}H_n(\tau_j) = \frac{1}{N} \sum_{k=1}^N \int_{\RR^N} h_{n_k}(\tau_j)= \frac{N}{N}\int_{\RR^N} \eta = \int_{\RR^N} \eta.
    \end{equation}
    Extracting a subsequence in $N$ (again, independent of $j$) we get that $H_N(\tau_j)$ converges to $\lambda_j$ in the vague topology. Using~\Cref{re:limpotential},
    \begin{equation}
        \liminf_{N \to \infty} \Psi_N(x,\tau_j) = (\mathcal{I}_\sigma*\lambda_j) (x) \quad \text{ quasi-everywhere}. 
    \end{equation}
    Since they are equal quasi-everywhere, we can use that $(u-v_s)(\tau_j)\phi_R  \in L^1\cap L^\infty(\RR^N)$ and~\Cref{re:cabsolutelyaluso} to prove that
    \begin{equation}
        \int_{\RR^N} (u(x,\tau_j)-v_s(x,\tau_j)) \phi_R(x) \Psi^j (x)\, dx = \int_{\RR^N} (u(x,\tau_j)-v_s(x,\tau_j)) \phi_R(x) \mathcal{I}_\sigma*\lambda_j\, dx.
    \end{equation}

    With this we can finish taking the limits in $n$ and $R$ in~\eqref{limitenR} to conclude the argument. We have arrived to the following expression
    \begin{equation}
        \int_{\RR^N} (u(x,T)-v_s(x,T)) \theta(x) \, dx =  \int_{\RR^N} (u(x,\tau_j)-v_s(x,\tau_j)) \Psi^j(x) \, dx.
    \end{equation}
    As we mentioned, $\{\Psi^j\}_j$ have the monotonicity in time inherited from the one for~$\{\psi_j\}_j$, explained in~\Cref{re:auxequations},
    \begin{equation}
        \label{monotonia2}
        0\leq \Psi^{j+1}(x) \leq \Psi^j(x) \leq \theta(x) \text{ for a.e } x \in \RR^N.
    \end{equation}
    Since $\Psi^j = \mathcal{I}_\sigma*\lambda_j$ quasi-everywhere, which is dense, and we have that the potentials satisfy 
    \begin{equation}
        (\mathcal{I}_\sigma*\lambda_j) (x) = \liminf_{y \to x} (\mathcal{I}_\sigma*\lambda_j) (y),
    \end{equation}
    see~\Cref{re:mark}, then
    \begin{equation}
        \label{monotonia3}
        0\leq (\mathcal{I}_\sigma*\lambda_{j+1}) (x) \leq (\mathcal{I}_\sigma*\lambda_j) (x) \leq \theta(x) \text{ for all } x \in \RR^N.
    \end{equation}
    The following conservation of mass condition is also available
    \begin{equation}
        \label{consmasa}
        \int_{\RR^N} d\lambda_j = \int_{\RR^N} \eta(x) \, dx.
    \end{equation}
    From~\eqref{monotonia3}, using~\Cref{re:potentialmonotoneconv}, we deduce that there exists a positive measure $\lambda_\infty$ such that
    \begin{gather}
        \lim_{j \to \infty} \lambda_j = \lambda_\infty \quad \text{ in the vague topology},\\
        \Psi^\infty(x) \coloneq \lim_{j \to \infty} (\mathcal{I}_\sigma*\lambda_j) (x)  = (\mathcal{I}_\sigma * \lambda_\infty)(x),\quad \text{ quasi-everywhere}.
    \end{gather}
    Through~\Cref{re:monotoniapotencial}, we have monotonicity in time of the potentials $U(x,t) = (\mathcal{I}_\sigma*u)(x,t)$. With Fubini's Theorem
    \begin{equation}
        \begin{aligned}
            \int_{\RR^N} U(x,t) d\lambda_j(x) & = \int_{\RR^N} \left( \int_{\RR^N} \frac{u(y,t)}{|x-y|^{N-\sigma}}\, dy  \right) d\lambda_j(x) \\
                                                     & =\int_{\RR^N} \left( \int_{\RR^N} \frac{u(y,t)}{|x-y|^{N-\sigma}}d\lambda_j(x)  \right) dy     \\
                                                     & = \int_{\RR^N} u(x,t) \Psi^j(x) \, dx.
        \end{aligned}
    \end{equation}

    Now, using the monotonicity of $U(x,t)$, we get for $k \leq j$ that
    \begin{equation}
        \int_{\RR^N} u(x,\tau_j) \Psi^j(x) \, dx =\int_{\RR^N} U(x,\tau_j)\,  d\lambda_j(x) \geq   \int_{\RR^N} U(x,\tau_k)\,  d\lambda_j(x)=\int_{\RR^N} u(x,\tau_k) \Psi^j(x) \,dx.
    \end{equation}
    Taking $\liminf_{j \to \infty}$ we obtain 
    \begin{equation}
        \liminf_{j \to \infty}\int_{\RR^N} U(x,\tau_j) \, d\lambda_j(x) \geq   \int_{\RR^N} U(x,\tau_k) \, d\lambda_\infty(x).
    \end{equation}
    Since this is true for every $k$ and $U$ converges monotonically to $\mathcal{I}_\sigma*\mu$,
    \begin{equation}\label{eq:firstinequality}
        \liminf_{j \to \infty}\int_{\RR^N} u(x,\tau_j) \Psi^j(x) \, dx \geq \int_{\RR^N} (\mathcal{I}_\sigma*\mu)(x) \, d\lambda_\infty(x) .
    \end{equation}
    It is important to point out that in this last step we require monotone convergence for every point, not just an almost everywhere convergence. This is because the measure $\lambda_\infty$ and the Lebesgue measure a priori may not have the same sets of measure zero.

    Recall now that $v_s(x,t) = v(x,t+s)$. The potential $V$ of $v$ satisfies the same condition as the potential of $u$, that is, it decreases as time increases, so the last inequality is verified thanks to
    \begin{equation}
        V_s(0) = V(s) \leq V(0) = \mathcal{I}_\sigma*\mu.
    \end{equation}
    Thus,
    \begin{equation}\label{eq:secondinequality}
        \begin{aligned}
            \limsup_j \int_{\RR^N} v_s(x,\tau_j) \Psi_j \, dx = \limsup_j \int_{\RR^N} V_s(x,\tau_j) \, d\lambda_j  \leq\limsup_j \int_{\RR^N} V_s(x,0) \, d\lambda_j & \\
                                                                                                                                  =  \limsup_j \int_{\RR^N} v_s(x,0) \Psi_j \, dx  = \int_{\RR^N} v_s(x,0) \, \Psi_\infty \,dx = \int_{\RR^N} V_s(x,0) \, d\lambda_\infty.&
        \end{aligned}
    \end{equation}

    The main reason why the translation in $s$ was introduced is to use Fubini's Theorem in the previous computation. This may not be possible for $s=0$. The monotonicity of $\{\Psi^j\}_j$ is also used here.

    We now join the last obtained inequalities~\eqref{eq:firstinequality} and~\eqref{eq:secondinequality} to get
    \begin{equation}
        \liminf_j \int_{\RR^N} (u-v_s)(\tau_j) \Psi_j \, dx \geq \int_{\RR^N} [(\mathcal{I}_\sigma*\mu)(x) - V_s(x,0)] \, d\lambda_\infty (x) \geq 0. 
    \end{equation}
    Therefore, we obtained from~\eqref{limitenR} that
    \begin{equation}
        \int_{\RR^N} (u-v_s)(T) \theta \geq 0.
    \end{equation}
    Since $v_s$ is a translation in time of $v$, we may take the limit $s \to 0$ since translations are continuous operators in $L^1$ and both $u,v$ are in $L^1(\RR^N \times (0,T))$.
    Finally, we got that
    \begin{equation}
        \int_{\RR^N} (U-V)(T) \eta \geq 0,
    \end{equation}
    so $U\geq V$ almost everywhere. Reversing the roles of $u$ and $v$, since $\eta$ was arbitrary, we obtain $U (T) = V (T)$ almost everywhere. Then, using~\Cref{re:potentialuniqueness} we conclude that $u (T) = v(T)$ almost everywhere.
\end{proof}

\begin{Remark} \label{re:slowdifchanges}
    The previous result works without almost any change for $m>1$. We just need to perform slightly different computations for the estimates found in the integrals $I_n^3(R)$ and $I_n^4(R)$.     
\end{Remark}

Let us prove now what is the main result of this paper. In it,~\Cref{re:firstuniqueness},~\Cref{re:measurecomparison} and~\Cref{re:orderedcomparison} are fundamental, since the proof relies on using weak solutions as an approximation and comparing them to weighted solutions.

\begin{Theorem} \label{re:mainuniqueness}
    Let $m_c < m <1$ and $\mu \in \mathfrak{M}_\sigma$. Then, there exists a unique weighted solution $u$ of~\eqref{equation} that takes $\mu$ as its initial condition.
\end{Theorem}
\begin{proof}
    Let us define $u$ as the weighted solution of~\eqref{equation} from initial condition $\mu$, which exists due to~\Cref{re:existencia}. Choose a sequence of smooth compactly supported functions $\{h_n\}_n$ such that~$h_n \leq h_{n+1}$ and $\lim_{n\to \infty} h_n = 1$. Let $w_1$ be the solution with initial condition $h_1 \mu$, which exists due to~\Cref{re:existenciafinita}. Now, define $w_{2}$ as the weak solution with initial data $h_{2}\mu$ such that~$w_1 \leq w_2 \leq u$, which exists thanks to~\Cref{re:measurecomparison}. Iterating this process, we have a monotone sequence $\{w_n\}_n$ of weak solutions satisfying $w_1 \leq w_2 \leq \cdots \leq u$. Bear in mind that, thanks to~\Cref{re:firstuniqueness}, we now know that these weak solutions $w_n$ are unique.   
    
    Define $w_\infty \coloneq \lim_n w_n$. We can easily check that $w_\infty$ is a weighted solution using the fact that~$w_\infty \in L_\text{loc}^\infty((0,\infty) \! : \! L^1_\vartheta)$, since $w_\infty \leq u$. Due to them being ordered solutions,~\Cref{re:orderedcomparison} is available and we get
    \begin{equation}
        \left( \int_{\RR^N} (u-w_\infty)(x,t) \vartheta_R(x) \, dx\right)^{1-m} \leq \left( \int_{\RR^N} (u-w_\infty)(x,\tau ) \vartheta_R(x) \, dx\right)^{1-m} + \frac{C |t-\tau| }{R^{\sigma - N(1-m)}}.
    \end{equation}
    From this, observe that if
    \begin{equation} \label{eq:finishingstep}
        \lim_{\tau \to 0 }\int_{\RR^N} w_\infty(x,\tau ) \vartheta_R(x) \, dx = \int_{\RR^N} \vartheta_R \, d\mu,
    \end{equation}
    then $u \leq w_\infty$ a.e.\ and thus $u = w_\infty$ almost everywhere. Similarly to previous arguments, using~\eqref{eq:salto_tiempo} with $\vartheta_R$ instead of $\phi \in \cont^\infty_c(\RR^N)$,
    \begin{equation}
        \left| \int_{\RR^N} w_n(t) \vartheta_R dx - \int_{\RR^N} h_n \vartheta_R d\mu \right| \leq \int_0^{t} \int_{\RR^N} w_n^m \vartheta_R^m \frac{|(-\Delta)^{\sigma/2}\vartheta|}{\vartheta_R^m},
    \end{equation}
    and, just as we did in previous instances, taking first the limit in $n\to \infty$ and then $t \to 0$ we obtain~\eqref{eq:finishingstep}. We have obtained that $u$ coincides with the monotone limit $w_\infty$ of the sequence~$\{w_n\}_n$. This sequence only depends on the initial data $\mu$, and its elements are unique thanks to~\Cref{re:firstuniqueness}. Therefore, any other weighted solution starting from initial condition $\mu$ would also coincide with~$w_\infty$.
\end{proof}

\section{Comments and possible extensions} \label{se:extensions}

\textsc{The case $N\leq \sigma$:} We have left open the case $N=1$, $\sigma \in [1,2)$. In order to prove uniqueness, we would require a different approach to~\Cref{re:firstuniqueness}, without using the Riesz potential.

\textsc{More general initial conditions:} The biggest question left open from this work is whether the theory developed here applies to very weak solutions that are not weighted solutions. Let us restrict ourselves to using measurable functions as initial data. Intuitively, the possible set of initial conditions for very weak solutions of~\eqref{equation} should be measurable functions $u_0$ verifying 
\begin{equation} \label{eq:intuitiveinitialcondition}
    \int_{\RR^N} \frac{u_0^m(x)}{(1+|x|^2)^{\frac{N+\sigma}{2}}}\, dx<\infty
\end{equation}
for $0<m<1$. We know that, as long as the initial condition $u_0$ satisfies 
\begin{equation} \label{eq:actualinitialcondition}
    \int_{\RR^N} u_0 \vartheta \, dx < \infty
\end{equation}
for $\vartheta \in \varTheta$, the very weak solution starting from this initial data is a weighted solution. But there could be functions $u_0$ satisfying~\eqref{eq:intuitiveinitialcondition} and not~\eqref{eq:actualinitialcondition}. The main problem is that we require these weights to be regular enough so that $(-\Delta)^\sigma \vartheta(x)$ makes pointwise sense while also satisfying
\begin{equation} \label{eq:quotientcondition}
    \int_{\RR^N} \left(\frac{|(-\Delta)^{\sigma/2}\vartheta(x)|}{\vartheta(x)^m} \right)^{\frac{1}{1-m}} < \infty.
\end{equation}
This is a technical restriction that only allows us to state that, for the moment, weighted solutions are subclass of very weak solutions and might not be the whole set. Understanding the set of functions that satisfy~\eqref{eq:quotientcondition} seems like a delicate problem that would allow us to answer this question.   

\textsc{The operator:} The operator used in this work is $(-\Delta)^\sigma$, the fractional Laplacian. We could easily adapt all the previous results, except those concerning uniqueness, to a more general divergence-form operator defined as
\begin{equation} \label{operator_def}
    \mathcal{L}f(x) = P.V. \int_{\mathbb{R}^N} (f(x)-f(y)) J(x,y)\, dy,
\end{equation}
where $P.V.$ is the Cauchy principal value, and $J$ is a measurable kernel such that
\begin{equation}\label{nucleo}
\tag{H$_J$}
\left\{\begin{aligned}
    & J(x,y) = J(y,x), \quad x,y \in \mathbb{R}^N, \, x \neq y,       \\
    & \frac{ 1 }{\Lambda |x-y|^{N+\sigma}} \leq J(x,y) \leq \frac{\Lambda}{|x-y|^{N+\sigma}}, \quad x,y \in \mathbb{R}^N, \, x \neq y, 
                                                                        \\
\end{aligned} \right.
\end{equation}
for some constants $\sigma\in (0,2)$ and $\Lambda >0$. The problem is that there were many results in~\Cref{se:uniqueness} that greatly benefited from using the fractional Laplacian specifically. This was due to the use of the Riesz potential. If analogous results could be found for the inverse of these more general operators $\mathcal{L}$, then it would be possible to extend these results to those operators. It might also be interesting to study this problem for the operators found in the very recent paper~\cite{WidderIrene}.

\textsc{The nonlinearity:} Another possible path to expand the previous results is to generalize the nonlinearity used here, $\varphi(s) = s^m$, to an increasing, non-Lipschitz function satisfying $\varphi(0) = 0$ and maybe some other additional conditions, just as the ones found in~\cite{DahlbergKenigFast}.

\textsc{Growing solutions for the slow diffusion equation:} In~\cite{BonforteVazquezQuantitativeEstimates} and in the present paper we consider solutions that may grow at infinity for the fractional fast diffusion equation. Recently, growing solutions for the fractional $p$-Laplacian for $1<p<2$ have also been studied, see~\cite{VazquezgrowingPLaplacian}. For the fractional slow diffusion equation, we have not found anything in the literature regarding this subject. Considering that in the local slow diffusion case growing solutions have been extensively studied, see~\cite{VazquezPorousMedium}, this is very interesting territory yet to be explored.  

\textsc{Smoothing effects:} In~\cite{HerreroPierreFastRegularizing} (and for more general nonlinearities in~\cite{DahlbergKenigFast}), a local smoothing effect for the local fast diffusion equation is proven when $m>m_c$. Solutions become instantaneously~$L^\infty_\text{loc}$ starting from $L^1_\text{loc}$ initial data. This is not studied in this work. It would be interesting to check if weighted solutions satisfy a similar estimate. 

\section*{Acknowledgements}
I am immensely grateful to my PhD advisors, \textsc{A. de Pablo} and \textsc{F. Quirós}, for their constant support, helpful advice and fruitful discussions on the topic.

I would also like to thank \textsc{M. Bonforte} for his useful feedback, which enhanced the clarity of this work. 

This research has been supported by grants PID2023-146931NB-I00, CEX2023-001347-S and RED2022-134784-T, all of them funded by MCIU/AEI/10.13039/501100011033.


\end{document}